\documentclass[10pt]{amsart}

\usepackage{amsmath,amsfonts,amssymb,amsthm,epsfig,color,esint}
\newtheorem{theorem}{Theorem}[section]
\newtheorem{lemma}[theorem]{Lemma}
\newtheorem{proposition}[theorem]{Proposition}
\newtheorem{corollary}[theorem]{Corollary}
\theoremstyle{definition}
\newtheorem{definition}[theorem]{Definition}

\numberwithin{equation}{section}

\newcommand{\R}{\mathbb{R}}

\newcommand{\dist}{{\rm dist}}
\newcommand{\eps}{\varepsilon}

\newcommand{\F}{{\mathcal{F}}}
\newcommand{\Ha}{{\mathcal{H}}}

\newcommand{\diver}{\operatorname{div}}
\newcommand{\trace}{\operatorname{Tr}}
\newcommand{\tang}{\operatorname{Tan}}

\renewcommand{\MR}[1]{\null}

\begin{document}

\title[Minimality via second variation for microphase]{Minimality via second variation for microphase separation of diblock copolymer melts}

\author{Vesa Julin}
\address{Department of Mathematics and Statistics, University of Jyv\"askyl\"a, Finland}
\email{vesa.julin@jyu.fi}

\author{Giovanni Pisante}
\address{Dipartimento di Matematica, Seconda Universita di Napoli, Caserta, Italy}	
\email{giovanni.pisante@unina2.it}

\date{\today}

\begin{abstract}
We consider a non-local isoperimetric problem arising as the sharp interface limit of the Ohta-Kawasaki free energy introduced to model microphase separation of diblock copolymers. We perform a second order variational analysis that allows us to provide a quantitative second order minimality condition.  We show that critical configurations with positive second variation are indeed strict local minimizers of the nonlocal perimeter. Moreover we provide, via a suitable quantitative inequality of isoperimetric type, an estimate of the deviation from minimality for configurations close to the minimum in the $L^1$ topology.  

\end{abstract}

\maketitle

\tableofcontents

\section{Introduction}

In this note we are interested in performing a second order analysis for a  non-local isoperimetric problem arising as a variational limit of the Ohta-Kawasaki functional introduced for a density functional theory for microphase separation of A/B diblock copolymers. 

Among the several mean field approximation theories proposed to model the phase separation in diblock copolymer melts, the one derived by Ohta and Kawasaki in \cite{OK} turns out to be one of the most promising from the mathematical point of view. Let $\Omega \subset \R^n$ be a bounded domain representing the volume occupied by the polymeric material. The free energy can be written as a non-local functional of Cahn-Hilliard type as
\begin{equation} \label{nonsharp}
\mathcal{E}_{\eps}(u)= \eps \int_{\Omega} |\nabla u|^2\, dx + \frac{1}{\eps} \int_{\Omega} (u^2-1)^2\, dx+ \gamma_0 \int_{\Omega} \int_{\Omega} G(x,y)( u(x)- m)( u(y)- m)\, dxdy,
\end{equation}
where $u\in H^1(\Omega)$ represents the density distribution of the monomers forming the copolymers, $G$ is the Green's function of the Laplace operator with Neumann boundary conditions and 
\[
m:=\fint_{\Omega} u \,dx
\] 
is the difference of the phases' volume fractions. The formulation \eqref{nonsharp} was first introduced in \cite{NO} and for a derivation of the Ohta-Kawasaki density function theory from the self-consistent mean field theory we refer the reader to \cite{CR} and the references therein. 

As pointed out in \cite{CR}, depending on the molecular structure of the polymers there are several regimes of phase mixture.  Nevertheless the presence of an observable phase separation occurs in the so called intermediate or strong segregation regimes, wherein the domain size is much larger than the interfacial length. This suggests that most of the features of the model can be described, from a mathematical point of view, by looking at the sharp interface limit  of  $\mathcal{E}_\eps$ as the thickness $\eps$ of the diffuse interface tends to zero. Thus we are lead to study the minimizers of the following energy functional, that arises as the $\Gamma$-limit of $\mathcal{E}_\eps$ in the $L^1$ topology,
\begin{equation} \label{sharp}
\mathcal{E}(u)= \frac{8}{3} |D u|(\Omega) +  \gamma_0 \int_{\Omega} \int_{\Omega} G(x,y)( u(x)- m)( u(y)- m)\, dxdy
\end{equation}
where $u : \Omega \to \{ -1, 1\}$ is a function of bounded variation and $|D u|(\Omega)$ denotes its total variation in $\Omega$. The functional \eqref{sharp} can be described in a more geometric way that turns out to be suitable for our analysis. Indeed, identifying the function $u$ with the set $E = \{ x \in \Omega \mid u(x)= 1\}$ and using the properties of $G(x,y)$,  we can rewrite $\mathcal{E}(u)$, up to a multiplicative constant, as
\begin{equation}
\label{functional}
J(E)= P(E, \Omega) + \gamma \int_{\Omega} |\nabla v_E|^2\, dx,
\end{equation}
where $P(E, \Omega)$ stands for the perimeter of $E$ in $\Omega$ and $v_E$ is the solution of 
\begin{equation}
\label{Laplace}
\begin{cases}
- \Delta v_E = u_E - m &\text{ in } \Omega, \\
\frac{\partial v_E}{\partial \nu} = 0 &\text{ on } \partial \Omega, \\
\int_{\Omega}v_{E}=0 \;,&
\end{cases}
\end{equation}
with the positions  
\[ 
u_E = \chi_E - \chi_{\Omega\setminus E} \;\;\;\text{ and } \;\;\;\; m = \fint_{\Omega} u_E \, dx.
\]  
We are thus lead to study the variational problem
\begin{equation}\label{problem}
\min \left\{  J(E) \mid E \subset \Omega , \, \fint_{\Omega} u_E \, dx = m \right\}.
\end{equation}
Let us note that the condition $ \fint_{\Omega} u_E \, dx = m$ in the problem \eqref{problem} is nothing but a volume constraint imposed on the admissible sets, indeed, by definition of $m$, it is equivalent to $|E| = \frac{1}{2}(m + 1)|\Omega|$.
 

The problem \eqref{problem} has been presented in \cite{ChSt2} as a mathematical paradigm for the phenomenon of energy-driven pattern formation associated with competing short and  long-range interactions. Recently an increasing interest has been devoted to the second order variational analysis of energy functionals which exhibit this competing behaviour. Indeed it has been successfully applied to prove stability and minimality criteria in several contexts (see for instance \cite{AFM,CJP}). Directly related to our problem is the work by Ren and Wei and by Choksi and Sternberg  (cfr. \cite{RW1,RW2,RW3,RW4,RW5,ChSt2}). In a series of papers, as a first step toward the validation of the conjectured periodicity of global minimizers of \eqref{functional},  they calculate the second variation of \eqref{functional} at a critical configuration and, among several other applications, they construct examples of periodical critical configurations and find conditions under which their second variation is positive definite. 

A different approach is taken in \cite{AFM}, where the authors consider the problem of minimizing the functional \eqref{functional} in the periodic case, i.e., when $\Omega$ is  the $n$-dimensional flat torus. They prove the local minimality of regular critical configurations whose second variation is  positive definite. In this direction goes also our present work, whose aim is to prove that a similar minimality criterion holds true  for the problem \eqref{problem}. 

In order to describe our main result, we recall that, if $E$ is a sufficiently regular critical set for $J$, then the following Euler-Lagrange equation holds:
\begin{equation} \label{euler_lag}
H_M + 4 \gamma v_E = \lambda  \quad \text{on } M,
\end{equation}
where $M:= \overline{\partial E \cap \Omega}$ is the relative boundary of $E$ in $\Omega$, $H_M(x)$ indicates the sum of the principal curvatures of $M$ and the constant $\lambda$ is a Lagrange multiplier associated to the volume constraint. Moreover a regular critical set $E$ turns out to meet $\partial \Omega$ orthogonally in the sense that 
\begin{equation} \label{orthogonality}
\langle \nu_M, \nu_{\Omega} \rangle = 0 \qquad  \text{on } \, M \cap \partial \Omega, 
\end{equation}
where $\nu_M$ is the outward unit normal to $M$ and $\nu_{\Omega}$ is the exterior normal to $\Omega$. Following \cite{ChSt2} and \cite{StZ} (cf. in particular \cite[Remark 2.8]{ChSt2}), the second variation of the functional \eqref{functional} at a regular critical set $E$ (cf. Definition \ref{regular_crit}), can be associated with the quadratic form defined, over all functions  $\varphi \in H^1(M)$ satisfying the integral condition 
\[\int_{M}\varphi \, d \Ha^{n-1}= 0,\]
as 
 \begin{equation}
\label{quadratic}
\begin{split}
\partial^2 J(E)[\varphi] = & \int_{M} \left( |D_{\tau}\varphi |^2- | B_M |^2  \varphi^2 \right)\, d\Ha^{n-1}  -  \int_{M \cap \partial \Omega}  B_{\partial \Omega}(\nu_M,\nu_M) \varphi ^2  \,d \Ha^{n-2}\\
&+ 8 \gamma \int_{M}  \int_{M}  G(x,y) \varphi(x)  \varphi(y) \, d\Ha^{n-1}(x)d\Ha^{n-1}(y)\\
&+  4 \gamma \int_{M}  \langle \nabla v_{E}, \nu_M \rangle \varphi^2 \, d\Ha^{n-1},
\end{split}
\end{equation}
where $ B_{\partial \Omega}$ stands for the second fundamental form of $\partial \Omega$, i.e., the Hessian of the distance function from  the boundary, $ | B_M |^2$ is the sum of the squares of the principal curvatures of $M$ and $G : \Omega \times \Omega \to \R$ is the Green's function defined by the following Neumann boundary value problem
\[
\begin{cases}
- \Delta_y G(x, y) = \delta_x - \frac{1}{|\Omega|} & \text{ in } \Omega, \\
D_{x}G(x,y)\cdot \nu_{\Omega}(x) =0  & \text{for any } x \text{ on } \partial \Omega, \\
\int_{\Omega} G(x,y) \, dx = 0. & 
\end{cases}
\]
It follows from the results in  \cite{ChSt2} and \cite{StZ} that every regular local minimizer $E$ of \eqref{functional} satisfies 
\[
\partial^2 J(E)[\varphi] \geq 0 \qquad \text{for every }\, \varphi \in H^1(M) \,\, \text{with }\,  \int_{M}\varphi \, d \Ha^{n-1}= 0.
\]  
As we already mention, our main result deals with the opposite question. In other words, we study whether we may deduce the stability of a critical set $E$ from the positivity of the second variation. Our goal is  to prove that a regular critical point with positive second variation is indeed an isolated local minimum. This is achieved in the following theorem.
\begin{theorem} \label{mainthm}
Suppose $\Omega$ is $C^{4,\alpha}$-regular, for some $0<\alpha\leq1$, and that $E \subset \Omega$ with $ \int_{\Omega} u_{E} \, dx =m$ is a regular critical set of \eqref{functional} which satisfies 
\[
\partial^2 J(E)[\varphi] >0
\] 
 for every $\varphi \in H^1(M)$ with $\int_{M}\varphi \, d \Ha^{n-1}= 0$. Then $E$ is a strict local minimum for $J$ and there exist constants $c>0$ and $\delta>0 $ such that
\[
J(F) \geq  J(E) + c\,  |F \triangle E |^2 
\]
for every set  $F$ of finite perimeter in $\Omega$ with $ \int_{\Omega} u_{F} \, dx =m$ and $|F \triangle E | \leq \delta$.
\end{theorem}

In particular, this implies that $E$ is a local minimizer of $J$ and, in addition, provides a quantitative estimate of the deviation from the  minimality for sets near to $E$. We remark that our result holds also in the case $\gamma=0$ and thus  we cover the case of volume constrained isoperimetric problem in a regular domain $\Omega$. However, in this case it is rather restrictive to assume that the second variation is positive definite, and not just positive semi-definite.  

Going back to the functionals $\mathcal{E}_{\eps}$ defined in \eqref{nonsharp}, as we already observed, it is well known that, for $\eps\to 0$, they $\Gamma$-converge in the $L^1$ topology to $\mathcal{E}$.  As a corollary of the previous theorem we can thus obtain the following stability result in the spirit of \cite[Theorem 1.3]{AFM}.
\begin{corollary} \label{corollary1}
Assume $E$ is a regular critical point of $J$  with positive second variation and denote $u_E = \chi_E - \chi_{\Omega \setminus E}$. There exist $\eps_0>0$ and a family $\{ u_{\eps}\}_{\eps < \eps_0}$ of isolated local minimizers of $\mathcal{E}_{\eps}$  satisfying the constraint $\int_{\Omega} u_{\eps} \, dx = m$ such that $u_{\eps} \to u_E$ in $L^1(\Omega)$ as $\eps \to 0$.
\end{corollary}
 
Some comments are in order on the differences between the periodic case (studied in \cite{AFM}) and the Neumann boundary case. Indeed, if on one side working with the Neumann setting dispenses us from several technicalities needed to deal with the translation invariance of the functional $J$, on the other side new delicate arguments need to be introduced to deal with problems which arise when $E$ touches $\partial \Omega$. We remark that in \cite{AFM} the Neumann boundary case was considered only under a  restrictive assumption that the set $E$ does not intersect the boundary of $\Omega$.

Finally we outline the structure of the paper. In the next section we introduce the notation and recall some preliminary results. In Section \ref{Sec-regularity} we discuss some regularity results for $\Lambda$-minimizers of the area functional and in particular a stability result for the regularity (cfr. Theorem \ref{convergence}) which will play an important role in the proof of the main theorem. Section 4 is devoted to  lengthy calculations of the second variation formula for regular sets which  satisfy the orthogonality condition \eqref{orthogonality} but are not critical. The main difficulty in the calculations are due the perimeter term in \eqref{functional}. As in \cite{CMM}, where the general second variation formula of the perimeter in $\R^n$ is derived,  we have to find the right formula for the second variation of the perimeter,  but in our case we have to include the terms induced by the boundary $\partial \Omega$.  The proof of the main result, Theorem \ref{mainthm}, is presented in sections 5 and 6. The scheme follows a well established path (see for instance \cite{AFM,CJP}). First we use the general second variation formula from Proposition \ref{2ndVariation} to prove the local minimality among regular sets which are close to the critical one in the $W^{2,p}$-topology and satisfy the orthogonality condition \eqref{orthogonality} (cf, Proposition \ref{smooth} for a precise statement).  The proof of the  $W^{2,p}$-minimality differs from the periodic case, since we always have to preserve the orthogonality condition (see Lemma \ref{the_flow}). The final result is then proved by contradiction using a penalization argument and exploiting the regularity theory of $\Lambda$-minimizers of the area functional.

\section{Preliminaries}

In this section we set up the basic notation and recall some preliminary results. 
Throughout the paper $\Omega \subset \R^n$ is assumed to be a bounded domain with boundary of class $C^{4, \alpha}$ for some $\alpha \in (0,1]$, unless otherwise mentioned. We say that a set $E \subset \Omega$ is \emph{$C^{k, \alpha}$-regular}, with $k\in \mathbb{N}$ and $\alpha\in (0,1]$, if its \emph{relative boundary}, defined by 
\begin{equation*} 
M:= \overline{\partial E \cap \Omega}
\end{equation*}
is a $C^{k,\alpha}$-regular $(n-1)$-dimensional manifold with or without boundary if  $\partial E \cap \partial \Omega \not = \emptyset $ or $\partial E \cap \partial \Omega = \emptyset $  respectively. 

For any sufficiently regular set  $E \subset \Omega$ such that $\partial M:= M \cap \partial \Omega \not = \emptyset$, we denote, for $x\in \partial M$,  by $\nu^*(x)$ the \emph{outward unit co-normal}  of $\partial M$ at $x$, i.e., the unit vector which is normal to $M \cap \partial \Omega$ and tangent to  $M$ at the point $x$. Given an oriented regular manifold $\mathcal{M}$ and $x\in \mathcal{M}$ we denote by $\nu_\mathcal{M}(x)$ the outward unit normal to $\mathcal{M}$ at the point $x$ and we use the notation $H_\mathcal{M}(x)$ to indicate the sum of the principal curvatures of $\mathcal{M}$ at $x$, i.e., $H_\mathcal{M}(x)= \diver_\tau \nu_\mathcal{M}$, where $\diver_\tau$ denotes the tangential divergence on $\mathcal{M}$ (see \cite[Section 7.3]{AFP}). We shall omit the dependence on the point $x$ whenever it is clear from the context.

Let $E$ be a $C^1$-regular set and let  $X$  be a $C^1$-vector field defined in $\overline{\Omega}$, which satisfies 
\begin{equation}
\label{tangent.cond}
X(x) \in \tang_x(\partial \Omega)\quad \text{for every } x \in \partial \Omega,
\end{equation}
where $\tang_x(\partial \Omega)$  denotes the tangent space of $\partial\Omega$ at the point $x$. We assume further  that
\begin{equation}
\label{volume1}
\int_{E} \diver (X) \, dx = \int_{M } \langle X, \nu_M \rangle \, d \Ha^{n-1} =0.
\end{equation}
We say that a vector field $X$ is \emph{admissible} if it satisfies \eqref{tangent.cond} and \eqref{volume1}. Using the flow associated to $X$, i.e., the local solution $\Phi : \overline{\Omega} \times (-t_0, t_0) \to \overline{\Omega}$ of the following Cauchy problem    
\begin{equation}\label{flow}
\frac{\partial}{\partial t} \Phi(x,t) = X(\Phi(x,t)), \qquad \Phi(x, 0)= x,
\end{equation}
we define the perturbations of $E$, induced by $X$ and parametrized by $t$, by 
\[
E_t= \Phi(E,t).
\] 
The condition \eqref{tangent.cond} ensures that the flow is well defined and does not leave $\Omega$, while  \eqref{volume1} ensures that $|E_t| = |E| + o(t)$. We then define the first variation of the functional \eqref{functional} at $E$ in the direction of the field $X$ (or, equivalently, with respect to the flow $\Phi$) by 
\[
\frac{d}{d t} J(E_t) \big|_{t=0}.
\]
The well known formula for the first variation (see \cite{ChSt2} and \cite{StZ}) leads us to the following definition of regular critical sets of \eqref{functional}. 


\begin{definition}
\label{regular_crit}
We say that  a $C^1$-regular set $E \subset \Omega$ is \emph{a regular critical set} for $J$ if for every $X \in C^{1}(\overline\Omega;\R^{n})$, which satisfies  \eqref{tangent.cond} and  \eqref{volume1}, it holds  
\begin{equation} \label{euler_lag-weak}
\int_{M}  \diver_\tau X \, d \Ha^{n-1}  + 4\gamma \int_{M} v_{E} \langle X , \nu_{M} \rangle d\Ha^{n-1} = 0.
\end{equation}
In other words, $E$ is a set for which the first variation of $J$ is zero in the direction of every admissible vector field $X$ of class $C^1$.
\end{definition}


Since any regular critical set, $E$, by definition, satisfies the equation \eqref{euler_lag} in the weak sense, we may use this information to infer higher regularity for $E$.  According to \cite[Remark 2.3]{AFM} and \cite[Section 1.3]{CS} it has been observed that a regular critical set is indeed of class $C^{3,\alpha}$ in $\Omega$ for any $\alpha \in (0,1)$. In the next proposition, we show that this regularity can be significantly improved using standard regularity theory for elliptic equations.
   
\begin{proposition}
\label{int_regularity}
Let $\Omega\subset \R^n$ be a domain with boundary of class $C^{k, \alpha}$ and $E \subset \Omega$ a regular critical set. Then  $E$ is $C^{k, \alpha}$-regular  and  $C^{\infty}$-regular in $\Omega' $ for every $\Omega'  \subset \subset \Omega$.  
\end{proposition}

\begin{proof}
We will prove in detail only the interior regularity since the regularity up to the boundary follows using a similar argument.  As we already observed, $\partial E$ is in fact of class $C^{3, \alpha}$ in $\Omega$ for every $\alpha \in (0,1)$. This follows from  \eqref{euler_lag} using standard elliptic estimates, after noticing that $v_E$, being a solution of \eqref{Laplace}, is of class $C^{1,\alpha}(\Omega)$.

Let $x_0$ be a point on $\partial E \cap \Omega$.  Then there exist a neighbourhood $U$ of $x_0$ and a $C^{3, \alpha}$ diffeomorphism $\Phi$ from $\bar{U}$ to the unit ball $\bar{B}_1$ such that $\Phi(U \cap E) = B_1 \cap \{ y_n <0 \}$, $\Phi(U \cap \partial E) = B_1 \cap \{ y_n = 0\}$ and $\det D\Phi^{-1}=1$. Define $w(y) = v_E(\Phi^{-1}(y))$ and $h(y)= u_E(\Phi^{-1}(y))-m$. By \eqref{Laplace} $w$ solves an equation of the form
\[
- \diver(A(y)Dw) = h.
\]
in $B_1$ with $A=\{A_{ij}\}_{i,j=1,\dots,n }$ of class $C^{2,\alpha}$.
Since $h(y)$ is constant in every direction, except in the normal direction $y_n$, we may differentiate the above equation with respect to $y_i$ for every $i = 1, 2, ..., n-1$ and obtain
\[
- \diver(A(y)Dw_{y_i}) =  \diver(A_{y_i}(y)Dw).
\]
Since $Dw \in C^{\alpha}(B_1)$ and $A_{ii} \in C^{2, \alpha}(B_1)$, we obtain from the standard elliptic regularity theory (cf. \cite{GT}, chapter 8) that   $w_{y_i} \in C^{1, \alpha}(B_1)$. This implies that the tangential derivative $D_{\tau} v_E$ is locally $C^{2, \alpha}$ regular on $\partial E \cap \Omega$ and the Euler equation \eqref{euler_lag} in turn implies that   $\partial E \cap \Omega$ is locally of class $C^{4, \alpha}$. By iteration we then obtain the interior $C^{\infty}$ regularity. 
\end{proof}

At the end of the section we remark that  Definition \ref{regular_crit} implies the orthogonality condition \eqref{orthogonality}, i.e., that the relative boundary $M$ of a regular critical set meets $\partial \Omega$ orthogonally. Indeed, since by Proposition  \ref{int_regularity}, $M$ is $C^2$-regular, using the divergence theorem on $M$ we have, for any vector Field $X$, that
\begin{equation}\label{div-theorem}
\int_{M} \diver_\tau X \, d \Ha^{n-1} = \int_M H_{M} \langle X, \nu_M \rangle \, d \Ha^{n-1}  +  \int_{M \cap \partial \Omega }\langle X, \nu^* \rangle \, d \Ha^{n-2}.
\end{equation} 
Moreover $M$ is classical solution of  \eqref{euler_lag}, thus by Definition \ref{regular_crit} and \eqref{div-theorem}, we can write, for any $X$ which satisfies  \eqref{tangent.cond} and  \eqref{volume1},
\begin{equation*}
\begin{split}
\int_{M \cap \partial \Omega }\langle X, \nu^* \rangle \, d \Ha^{n-2} = &  \int_M (H_M + 4\gamma v_E)\langle X, \nu_M \rangle \, d \Ha^{n-1}  +  \int_{M \cap \partial \Omega }\langle X, \nu^* \rangle \, d \Ha^{n-2} \\
= & \int_{M}  \diver_\tau X \, d \Ha^{n-1}  + 4\gamma \int_{M} v_{E} \langle X , \nu_{M} \rangle d\Ha^{n-1}  = 0
\end{split}
\end{equation*}
and the orthogonality condition follows (cf.  \cite[Proof of Theorem 2.2]{StZ}).

\section{Regularity of $\Lambda$-minimizers} \label{Sec-regularity}

In this section we briefly discuss the regularity of sets which we call $\Lambda$-minimizers of the perimeter. The regularity result will later play an important role in the proof of Theorem \ref{mainthm}, but  we will also use it to deduce partial $C^1$ regularity for  local minimizers of \eqref{functional}. Together with Proposition \ref{int_regularity} this will then imply partial $C^{\infty}$ regularity  for the local minimizers.

\begin{definition}
\label{almost_area}
A set $E$ of finite perimeter in $\Omega$, i.e. such that $\chi_E  \in BV(\Omega)$, is a \emph{$\Lambda$-minimizer} of the perimeter  if for every $G \subset \Omega$ it holds
\[
P(E, \Omega) \leq P(G, \Omega) + \Lambda | G \triangle E |.
\]
\end{definition}

The above definition appears e.g. in \cite{CL} where it is named \emph{strong $\Lambda$-minimality}. Very similar to this are the definitions of 'almost minimizer' or 'quasiminimizer' of the perimeter, used for example in \cite{Gr} and \cite{Tam}, where  the perturbation is assumed to take place in a small ball, i.e. $G \triangle E \subset B_r(x_0)$, and the measure of the symmetric difference $| G \triangle E |$ is replaced by the volume of that ball, $|B_{r}|$. The advantage of $\Lambda$-minimality is that it implies the existence of a uniformly bounded generalized mean curvature vector. This allows us to use the regularity result by Gruter and Jost \cite[Theorem 4.9]{GJ} to prove the regularity of $\Lambda$-minimizers. 

\begin{theorem} \label{gruter} 
Suppose that $E$ is a  $\Lambda$-minimizer in $\Omega$. Then for every $\alpha \in (0,1)$ the relative boundary $M:= \overline{\partial E \cap \Omega}$ of $E$ is $C^{1, \alpha}$-regular outside a singular set $\Gamma$ with Hausdorff dimension $\dim_{\Ha}(\Gamma) \leq n- 8$. Moreover $M$ meets $\partial \Omega$ orthogonally on $\partial \Omega \setminus \Gamma$. 
\end{theorem} 

\begin{proof}
The result follows essentially  from the work by Gr\"uter \cite{Gr}, where he proves the regularity for minimizers of the partitioning problem, i.e.  \eqref{functional} with $\gamma=0$. We need only to show that we may indeed apply the argument in  \cite{Gr}.

Let us first show that there exists a function ${\bf{H}}: M \to \R^n$ such that  
\begin{equation} \label{curv_equation}
\int_{M} \diver_\tau X \, d \mathcal{H}^{n-1} =  \int_{M} \langle X, {\bf{H}} \rangle \, d \mathcal{H}^{n-1} 
\end{equation}
for every $X \in C^1(\Omega)$ which satisfies the tangential condition \eqref{tangent.cond}, and  that $\|{\bf{H}}\|_{L^{\infty}} \leq \Lambda$. 

For a fixed vector field $X$ satisfying \eqref{tangent.cond}, consider the flow 
\[
\frac{\partial}{\partial t} \Phi(x,t) = X(\Phi(x,t)), \qquad \Phi(x, 0)= x.
\]
Since $E$ is an $\Lambda$-minimizer it holds
\begin{equation} \label{alm_min_prop}
P(E, \Omega) \leq P(\Phi(E, t), \Omega) + \Lambda |\Phi(E, t) \triangle E|
\end{equation}
for some small $t>0$. We claim that 
\[
|\Phi(E, t) \triangle E| \leq \int_0^t \int_{M}|\langle X, \nu _M \rangle| \, d \Ha^{n-1} ds + o(t).
\]

Indeed, by the standard mollifying argument we  find a sequence of smooth sets $E_{\eps}$ such that $D \chi_{E_{\eps}} \to D \chi_{E}$ weakly* in $\Omega$ and  
$|D \chi_{E_{\eps}}|(\Omega) \to |D \chi_{E}|(\Omega)$. Denote $M_{\eps}:= \overline{\partial E_{\eps} \cap \Omega}$. We may estimate the $n$-dimensional tangential Jacobian
 of $\Phi$ on $M_{\eps} \times [0,t)$ from above by  $|\langle X, \nu_{M_{\eps}} \rangle| + \eps(t)$, with a small remainder term $\eps(\cdot)$. Hence, by the  area formula, we have 
\[
\begin{split}
|\Phi(E_{\eps}, t) \triangle E_{\eps}| &= |\{ \Phi(x, s) \mid  x \in M_{\eps}, \, \, s \in [0,t)\} \\
&\leq  \int_0^t \int_{M_{\eps}}|\langle X, \nu_{M_{\eps}} \rangle| \, d \Ha^{n-1} ds + o(t).
\end{split}
\]
The above claim follows from Reshetnyak continuity theorem, \cite[Theorem 2.39]{AFP}, by letting $\eps \to 0$.

Finally letting $t \to 0$ in \eqref{alm_min_prop} we obtain
\[
\int_{M} \diver_\tau X \, d \Ha^{n-1} \leq  \Lambda \int_{M} |\langle X, \nu_M \rangle| \, d \Ha^{n-1}.
\]
Since the above inequality holds for every admissible $X$ the claim \eqref{curv_equation} follows from Riesz representation theorem.

We may thus apply the regularity result from \cite{GJ} and the $C^1$-regularity follows exactly as in \cite{Gr}. In particular,  this result implies  that,  if the excess   at a given point $x_0 \in M$ 
\[
\lambda(E, x_0, \rho):= \rho^{n-1}\left( \int_{B_{\rho}(x_0) \cap\Omega} |D \chi_E| - \big| \int_{B_{\rho}(x_0) \cap\Omega} D \chi_E \big|  \right)
\]
is small, then $x_0$ is a regular point. By  regular point we mean that  $M$ coincides with an oriented $C^1$-submanifold (with or without boundary) in a neighbourhood of $x_0$. The $C^{1, \alpha}$-regularity follows from the fact that the curvature is bounded. 
\end{proof}

It is well known that in dimension $n \geq 8$ the singular set, even of a minimal surface, might not be empty. However, as we pointed out in the proof of  Theorem \ref{gruter}, we have the full regularity, if we know a priori that the excess $\lambda(E, x, \rho)$ is small for every $x \in \Omega$.  This leads to the following classical convergence result. The proof  can be found e.g in \cite{CL} with  few modifications due to the fact that the relative boundary is a manifold with boundary.

\begin{theorem} \label{convergence}
Suppose that  $E_k$ are $\Lambda$-minimizers in $\Omega$. Suppose that $E$ has $C^{1,\alpha}$-boundary in $\Omega$ with $\alpha \in (0, 1)$. If $E_k \to E$ in $L^1$ then
\[
\partial E_k \to \partial E \quad \text{in }\,\, C^{1,\alpha}.
\]
In particular, every $E_k$ is $C^{1, \alpha}$-regular when $k$ is sufficiently large.
\end{theorem}

In order to exploit the relation between the local minimizers of the functional $J$ and the $\Lambda$-minimizers of the perimeter, we start recalling that there exists a dimensional constant $C$ such that for every pair of measurable sets $E,F \subset \Omega$ we have
\begin{equation} \label{lipschitz}
\left|  \int_{\Omega} |\nabla v_F|^2\, dx-  \int_{\Omega} |\nabla v_E|^2\, dx \right| \leq  C\,  |F \triangle E| \,,
\end{equation}
where $v_E$ and $v_F$ are defined by the problem \eqref{Laplace}. The previous property of Lipschitz continuity of the non-local part of the energy \eqref{functional} follows for example arguing as in 
\cite[Lemma 2.6]{AFM} (cf. also \cite[Section 2.1]{CS}). The next proposition allows us to apply the regularity theory for $\Lambda$-minimizers to local minimizers of $J$.

\begin{proposition}
\label{almost_area_2} 
Any local minimizer $E$ of the functional \eqref{functional} is a $\Lambda$-minimizer of the perimeter for some $\Lambda$ depending on $E$.
\end{proposition}

\begin{proof}
Following the same argument of \cite[Proposition 2.7]{AFM} 
we can show that, if $E$ is a local minimizer for $J$, there exist $\delta>0$ and $\lambda>0$ such that $E$ solves the following penalized minimization problem
\begin{equation}
\label{volum_penal}
\min\{ J(F) + \lambda \big| |F| - |E| \big| \, : \, F \subset \Omega ,\, \, |F \triangle E| < \delta  \}.
\end{equation}
Let $G \subset \Omega$, according to Definition \ref{almost_area}, we need to show that 
\[
P(E, \Omega) \leq P(G, \Omega) + \Lambda |G \triangle E|
\]
for some large $\Lambda$. If $|G \triangle E| \geq \delta$, the above inequality is trivially true. If $|G \triangle E| < \delta$, then \eqref{volum_penal} yields
\[
J(E) \leq  J(G) + \lambda \big| |G| - |E| \big|.
\]
The Lipschitz continuity of the non-local part \eqref{lipschitz} gives
\[
\begin{split}
P(E, \Omega) &\leq P(G, \Omega)  + \lambda \big| |G| - |E|  \big|  + C\gamma |G \triangle E| \\
&\leq  P(G, \Omega) + (\lambda+ C\gamma)  |G \triangle E|.
\end{split}
\] 
proving the claim.
\end{proof}

In view of Proposition \ref{almost_area_2} we can invoke the regularity for $\Lambda$-minimizers proved in Theorem \ref{gruter} and the Proposition \ref{int_regularity}, to infer that every local minimizer of \eqref{functional}  is locally $C^{\infty}$-regular outside a critical set $\Gamma$ with $\dim_{\Ha}(\Gamma) \leq n-8$. This motivates us to define regular critical sets as critical sets with no singularities.

\section{Second variation formula}


In this section we calculate the second variation of the functional \eqref{functional} and our aim is to write it in a form where the quadratic structure appears. As it has been done in \cite{AFM} for the periodic boundary conditions case, we generalize the formula from \cite{ChSt2} and calculate the second variation at any regular set $E$, not necessarily critical. 

We assume that $E \subset \Omega$ is $C^2$-regular. 
As in Section 2 we consider a vector field $X \in C^2(\overline{\Omega}; \R^n)$ which satisfies the tangential condition \eqref{tangent.cond} and the associated flow  $\Phi : \Omega \times (-t_0, t_0) \to \Omega$ defined by \eqref{flow}.
We define the second variation at $E$ in the direction of $X$ or with respect to the flow $\Phi$ as
\[
\frac{d^2}{d t^2} J(E_t) \big|_{t=0},
\]
where $E_t= \Phi(E,t)$. Notice that we have not yet imposed any condition on the volume of $|E_t|$.

A major technical challenge is the calculation that involves the local part of the energy, $P(\cdot , \Omega)$. We follow a slightly different path than \cite[Proposition 3.9]{CMM} where the general second variation formula of the perimeter in $\R^n$ is derived. Instead of differentiating the first variation formula 
we will first calculate the second  derivative of  $P(E_t, \Omega)$, as it is done e.g. in \cite{Giusti},  and then  use the divergence theorem  in order to find its ``correct'' form. This saves us from differentiating the boundary term
\[
 \int_{M \cap \partial \Omega }\langle X, \nu^* \rangle \, d \Ha^{n-2}
\]
which appears in the first variation formula. Along the calculations it will become clear that we need to assume the orthogonality condition \eqref{orthogonality} on $E$ in order to find the quadratic structure in the second variation at $E$. However, this is not too restrictive since, as we pointed out in the previous section, \eqref{orthogonality} is a necessary condition for local minimality (indeed even for $\Lambda$-minimality by Theorem \ref{gruter}), thus it can be viewed as a part of the regularity assumption on $E$.

Imposing the orthogonality condition on $E$ implies $\nu^* = \nu_{\Omega}$ on $M \cap \partial \Omega$ and, by \eqref{div-theorem}, we may write the integration by parts formula for any $f \in C^1(\Omega)$ as
\[
\int_{M} f \diver_\tau X \, d \Ha^{n-1} = -  \int_M \langle D_{\tau}f, X \rangle \,  d \Ha^{n-1} + \int_M H  f \langle X, \nu_M \rangle \,   d \Ha^{n-1} +  \int_{M \cap \partial \Omega } f \langle X, \nu_{\Omega} \rangle \, d \Ha^{n-2},
\]
where we recall that $D_{\tau}f = Df - \langle Df, \nu_M \rangle \nu_M$ is the tangential derivative on $M$.

\begin{proposition}
\label{2ndVariation}
Let  $E \subset \R^n$ be a $C^2$-regular set which satisfies the orthogonality condition \eqref{orthogonality} 
and $X \in C^2(\overline{\Omega}; \R^n)$ be a vector field satisfying the tangential condition  \eqref{tangent.cond}. Then the second variation of \eqref{functional} at $E$ in the direction of $X$ can be written as
\[
\begin{split}
\frac{d^2 J(E_t)}{dt^2} \big|_{t=0} = &\int_{M} \left( \big| D_{\tau}\langle X, \nu_M \rangle \big|^2- | B_M |^2  \langle X, \nu_M \rangle^2 \right)\, d\Ha^{n-1}  -  \int_{M \cap \partial \Omega}  B_{\partial \Omega}(\nu_M,\nu_M) \langle X , \nu_M  \rangle^2  \,d \Ha^{n-2} \\
&+ 8 \gamma \int_{M}  \int_{M}  G(x,y) \langle X(x), \nu_M \rangle   \langle X(y), \nu_M \rangle\, d\Ha^{n-1}(x)d\Ha^{n-1}(y)\\
&+  4 \gamma \int_{M}  \langle \nabla v_{E}, \nu_M \rangle \,   \langle X, \nu_M \rangle^2 \, d\Ha^{n-1} -\int_M (H + 4 \gamma v_E)\diver_{\tau} \left( X_{\tau} ( \langle X, \nu_M \rangle) \right) \, d \Ha^{n-1} \\
&+ \int_M (H + 4 \gamma v_E)\diver(X)  \langle X, \nu_M \rangle  \, d \Ha^{n-1}  .
\end{split}
\]
Here $ | B_M |^2$ is the sum of the square of the principal curvatures of $M$, $B_{\partial \Omega}$ stands for the second fundamental form of $\partial \Omega$, i.e., the Hessian of the distance function from  the boundary and  $X_{\tau} = X- X_{\nu}$ where $X_{\nu}= \langle X,\nu_M \rangle \nu_M$.  
\end{proposition}

Before giving the proof, we remark that  the above formula agrees with the quadratic form \eqref{quadratic} with two additional terms for $\varphi := \langle X, \nu_M \rangle$. It turns out that both of these extra terms vanish if $E$ is critical and the flow is volume preserving. Indeed, if $|E_t| = |E|$, we deduce, by standard calculations (cf. \cite[equations (2.29) and (2.30)]{ChSt2} or \cite[equations (2.15) and (2.17)]{StZ}), that
\begin{equation} \label{1st_vol}
0 = \frac{d}{dt}|E_t|  \big|_{t=0}  = \int_{E} \diver (X) \, dx = \int_{M } \langle X, \nu_M \rangle \, d \Ha^{n-1},
\end{equation}
which ensure \eqref{volume1} and 
\begin{equation} \label{2nd_vol}
0 =\frac{d^2}{dt^2}|E_t|  \big|_{t=0}  = \int_{E} \diver \left( \diver (X) X \right) \, dx = \int_{M} \diver (X)  \langle X, \nu_M \rangle \, d \Ha^{n-1}.
\end{equation}
Therefore, if $E$ is a regular critical set, since $H + 4 \gamma v_E$ is constant, for a volume preserving flow, we obtain
\[
\frac{d^2 J(E_t)}{dt^2} \big|_{t=0} = \partial^2 J(E)[\varphi], 
\]
where $ \partial^2 J(E)$ is the quadratic form \eqref{quadratic}.  



\begin{proof}[Proof of the Proposition \ref{2ndVariation}]
We analyze separately the perimeter and the non-local part of the energy for which we use respectively the notations 
\[
P(t) = P(E_t, \Omega), \qquad F(t) = \gamma \int_{\Omega} |\nabla v_{E_t}|^2 \, dx.
\]

The value of $F''(0)$ has been calculated in \cite{ChSt2} in the periodic setting. However, as it is pointed out in \cite[Remark 2.8]{ChSt2},  replacing the Neumann boundary condition does not produce any new terms to it and the formula for $F''(0)$ is the same as in the periodic case.  Recalling the condition \eqref{tangent.cond} we thus have (cf. \cite[Proof of Theorem 2.6, step 3]{ChSt2})
\begin{equation}
\label{2VarNonLoc}
\begin{split}
F''(0) =&  8 \gamma \int_{M}  \int_{M}  G(x,y) \langle X(x), \nu \rangle   \langle X(y), \nu \rangle\, d\Ha^{n-1}(x)d\Ha^{n-1}(y)\\
&+  4 \gamma \int_{M}  \diver (v_E X)\,   \langle X, \nu \rangle \, d\Ha^{n-1}.
\end{split}
\end{equation}
where for simplicity we have set (and we will continue to use this convention) $\nu = \nu_M$. Also the expression for $P''(0)$ is well known, indeed, from \cite[Theorem 10.4]{Giusti}, we have
\[
P''(0) = \int_{M} \left( \diver_\tau Z + (\diver_\tau X)^2 +  |(D_{\tau}X)^T \nu_M|^2-  \trace   (D_{\tau} X)^2 \right) \, d \Ha^{n-1},
\]
where $Z$ is the acceleration vector field given by  $Z = DX \cdot X$, $\diver_\tau X=\diver X -  \langle DX \nu, \nu \rangle$ and $D_{\tau} X=DX -  (DX \nu) \otimes  \nu$ i.e. $(D_{\tau} X)_{i,j}= \langle D_{\tau}X_j , e_i \rangle$. By the divergence theorem on $M$ we have
\begin{equation}
\label{E''}
P''(0) = \int_{M}  \left(  |(D_{\tau}X)^T \nu|^2 + (\diver_\tau X)^2 -  \trace  \left( (D_{\tau} X)^2 \right) + H \langle Z, \nu \rangle \right)\, d\Ha^{n-1} +   \int_{M \cap \Omega} \langle Z, \nu_{\Omega}\rangle \, d\Ha^{n-2},
\end{equation}
where we have used that $\nu^*= \nu_{\Omega}$ on $M \cap \partial \Omega$ by the orthogonality condition.  Notice that this is the second variation formula for the perimeter $P(E_t, \Omega) $ given by any vector field $X$. 

We have to manipulate \eqref{E''} in order to  find the quadratic structure in it. To this aim we will use the decomposition of $X$ in its tangential and normal component with respect to $M$, i.e. we write $X = X_{\nu} + X_{\tau}$ where $X_{\nu} = \langle X, \nu \rangle \nu$ and rewrite every term in \eqref{E''} acordingly.
Let us start with the first term. Since $D \nu \, \nu = 0$ we have $(D_{\tau} X_{\nu})^T \nu =  D_{\tau}\langle X , \nu  \rangle$ and therefore
\begin{equation}
\label{1st}
\begin{split}
|(D_{\tau}X)^T \nu|^2 &=  |(D_{\tau}X_{\nu})^T \nu|^2  +2 ((D_{\tau}X_{\nu})^T \nu ) \cdot ( (D_{\tau} X_{\tau})^T \nu) +|(D_{\tau}X_{\tau})^T \nu|^2 \\
&= | D_{\tau}\langle X , \nu  \rangle|^2 + 2 \,   D_{\tau}\langle X , \nu  \rangle  \cdot  ((D_{\tau} X_{\tau})^T \nu ) +|(D_{\tau}X_{\tau})^T \nu|^2 .
\end{split}
\end{equation}
Since $\diver_\tau X_{\nu} = H \langle X, \nu \rangle$  the second term in \eqref{E''} can be written as 
\begin{equation}
\label{2nd}
\begin{split}
 ( \diver_\tau X)^2 &=  (\diver_\tau X_{\nu}+  \diver_\tau X_{\tau})^2 \\
&= H^ 2 \langle X, \nu \rangle^ 2 + 2 H \langle X, \nu \rangle \diver_\tau X_{\tau }  +  ( \diver_\tau X_{\tau })^2  .
\end{split}
\end{equation}
For the third term in \eqref{E''} we will use the equalities $(D_{\tau} X_{\nu})^2= \langle X , \nu  \rangle^2 (D \nu)^2 + (D_{\tau} \langle X , \nu  \rangle \cdot \nu ) \, \nu \otimes D_{\tau}\langle X , \nu  \rangle$ and $ (D_{\tau} X_{\tau}) (D_{\tau} X_{\nu}) =  \langle X , \nu  \rangle D \nu\, D_{\tau} X_{\tau} $. Hence, we deduce
\begin{equation}
\label{3th}
\begin{split}
 \trace   (D_{\tau} X)^2 &=  \trace   (D_{\tau} X_{\nu})^2 +2 \trace   \big(  (D_{\tau} X_{\tau}) (D_{\tau} X_{\nu})\big) +\trace   (D_{\tau} X_{\tau})^2 \\
&= \langle X , \nu  \rangle^2 |B_M|^2  +2\langle X , \nu  \rangle \trace \big( D \nu\, D_{\tau} X_{\tau} \big) +\trace   (D_{\tau} X_{\tau})^2 .
\end{split}
\end{equation}
We treat the fourth term in \eqref{E''} by writing 
\begin{equation} \label{4th}
\begin{split}
 \langle Z, \nu \rangle = \langle DX X , \nu \rangle &=  \langle \left(  DX (X_{\nu}) + D (X_{\nu}) (X_{\tau}) +  D (X_{\tau}) (X_{\tau}) \right), \nu \rangle \\
 &=   \langle X , \nu \rangle  \langle DX \, \nu, \nu \rangle +  D_{\tau} \langle X, \nu \rangle \cdot X_{\tau} + \langle Z_{\tau}, \nu \rangle \\
&= \langle X , \nu \rangle  \diver X -  \langle X , \nu \rangle  \diver_\tau X +D_{\tau} \langle X, \nu \rangle \cdot X_{\tau}  + \langle Z_{\tau}, \nu \rangle \\
&= \langle X , \nu \rangle  \diver X -     \diver_\tau (X_{\tau} \langle X , \nu \rangle ) - H \langle X, \nu \rangle^2 + 2X_{\tau} \cdot D_{\tau} \langle X, \nu \rangle  + \langle Z_{\tau}, \nu \rangle ,
\end{split}
\end{equation}
where we have set $Z_{\tau} =  D (X_{\tau}) (X_{\tau})$ and in the last equality we have used $\diver_\tau X_{\nu} = H \langle X, \nu \rangle$.

For the last term $ \int_{M \cap \Omega} \langle Z, \nu_{\Omega} \rangle \, d\Ha^{n-2}$ we notice that the tangential condition \eqref{tangent.cond} on $X$  implies $\langle X, \nu_{\Omega} \rangle = 0$ on  $M \cap \partial \Omega$. In particular,   $D \langle X, \nu_{\Omega} \rangle \cdot X = 0$ on $M \cap \partial \Omega$. Hence the last term in \eqref{E''}  becomes
\begin{equation} \label{5th}
 \int_{M\cap \partial \Omega}    \langle Z, \nu_{\Omega} \rangle \, d \Ha^{n-2} =  \int_{M\cap \partial \Omega}  \langle DX X, \nu_{\Omega} \rangle \, d \Ha^{n-2} = - \int_{M\cap \partial \Omega}  \langle D \nu_{\Omega} X ,X \rangle \, d \Ha^{n-2}.
\end{equation}
From now on we will use the notation $\langle D \nu_{\Omega} X ,X \rangle = B_{\partial \Omega}(X, X)$ that has meaning again by \eqref{tangent.cond}.
We use \eqref{1st},  \eqref{2nd},  \eqref{3th},  \eqref{4th} and  \eqref{5th} to  rewrite \eqref{E''} as
\begin{equation}
\label{all_terms}
\begin{split}
P''(0) =& \int_{M} |D_{\tau} \langle X , \nu  \rangle|^2- |B_M|^2 \langle X, \nu \rangle^2 \, d \Ha^{n-1} -  \int_{M \cap \partial \Omega} B_{\partial \Omega}(X, X) \,d \Ha^{n-2}  \\
&+ 2 \int_{M}D_{\tau}\langle X , \nu  \rangle  \cdot  ((D_{\tau} X_{\tau})^T \nu )\, d\Ha^{n-1} + 2\int_{M}H \left( \langle X, \nu \rangle \diver_\tau X_{\tau } \right)  \, d\Ha^{n-1} \\
&- 2  \int_{M} \langle X , \nu  \rangle \trace \big( D \nu\, D_{\tau} X_{\tau} \big) \, d\Ha^{n-1}+  2  \int_{M} H \left( X_{\tau} \cdot D_{\tau} \langle X, \nu \rangle \right)\, d\Ha^{n-1} \\
&+  \int_{M}H \left(   2X_{\tau} \cdot D_{\tau} \langle X, \nu \rangle +  \langle X , \nu \rangle  \diver X  -     \diver_\tau (X_{\tau} \langle X , \nu \rangle )  \right)\, d \Ha^{n-1}  \\
&+ \int_{M} ( \diver_\tau X_{\tau})^2+  |(D_{\tau}X_{\tau})^T \nu|^2-  \trace   (D_{\tau} X_{\tau})^2 + H \langle Z_{\tau}, \nu \rangle  \, d\Ha^{n-1} .
\end{split}
\end{equation}

We continue by treating the mixed terms in \eqref{all_terms}.  We  integrate by parts the third term in \eqref{all_terms}  and deduce
\begin{equation}
\label{int_by_parts1}
\begin{split}
2\int_{M} D_{\tau}\langle X , \nu  \rangle  \cdot  ((D_{\tau} X_{\tau})^T \nu )\, d\Ha^{n-1}  = &- 2\int_{M} \langle X , \nu  \rangle  \diver_\tau  \big( (D_{\tau} X_{\tau})^T \nu \big)\, d\Ha^{n-1} \\
 & + 2\int_{M \cap \partial \Omega}\langle X , \nu  \rangle   \langle (D_{\tau} X_{\tau})^T \nu , \nu_{\Omega}\rangle \, d\Ha^{n-2}\\
= & - 2\int_{M} \langle X , \nu  \rangle  \diver_\tau  \big( (D_{\tau} X_{\tau})^T \nu \big)\, d\Ha^{n-1} \\
 & - 2\int_{M \cap \partial \Omega}\langle X , \nu  \rangle   \langle D \nu \,  X_{\tau}  , \nu_{\Omega}\rangle \, d\Ha^{n-2},
\end{split}
\end{equation} 
where we have used that  
\[
\langle D_{\tau} X_{\tau} \,  \nu_{\Omega} , \nu \rangle  +\langle D \nu \, X_{\tau}  , \nu_{\Omega} \rangle =D_{\tau} \langle X_{\tau}, \nu \rangle \cdot \nu_{\Omega}  =0.
\] 
For the fourth  term in \eqref{all_terms} we observe that, by the orthogonality condition, $X_{\tau}$ vanishes on $M \cap \partial \Omega$ and therefore integration by parts yields
\begin{equation}
\label{int_by_parts2}
\begin{split}
2 \int_{M} H \langle X, \nu \rangle \diver_\tau X_{\tau } \, d\Ha^{n-1} = &   - 2 \int_{M} D_{\tau} \left(  H \langle X, \nu \rangle \right)  \cdot X_{\tau}\, d\Ha^{n-1}  \\ 
& + 2 \int_{M \cap \partial \Omega} H \langle X , \nu  \rangle   \langle  X_{\tau} , \nu_{\Omega}\rangle \, d\Ha^{n-2} \\
= & - 2 \int_{M} H  \, X_{\tau} \cdot  D_{\tau }\langle X , \nu  \rangle\, d\Ha^{n-1}  \\
& - 2 \int_{M} \langle X , \nu  \rangle D_{\tau }H  \cdot X_{\tau}\, d\Ha^{n-1}.
\end{split}
\end{equation}

Next, using the notation $\delta_i$ for partial derivative on $M$, i.e., $\delta_i g= \partial_{x_i}g - \langle Dg , \nu \rangle \nu_i $ for any smooth function $g$, we use that  $ \delta_j \nu_i =  \delta_i \nu_j $ and $ \sum_{i,j}^n  (\delta_j \delta_i \nu_i ) X_{\tau}^{j} = \sum_{i,j}^n  (\delta_i \delta_i \nu_j ) X_{\tau}^{j} $ (see \cite[Lemma 10.7]{Giusti}) to obtain 
\begin{equation}
\label{mixed_terms}
\begin{split}
 \diver_\tau  \big( (D_{\tau} X_{\tau})^T \nu \big) &+ D_{\tau }H  \cdot X_{\tau}+ \trace \big( D \nu\, D_{\tau} X_{\tau} \big)  \\
&= \sum_{i,j}^n  \delta_i (\delta_i X_{\tau}^{j} \, \nu_j) + (\delta_j \delta_i \nu_i ) X_{\tau}^{j} + \delta_j \nu_i \delta_i X_{\tau}^{j} \\
&=  \sum_{i,j}^n  \delta_i (\delta_i X_{\tau}^{j} \, \nu_j) +  (\delta_i \delta_i \nu_j ) X_{\tau}^{j} + \delta_i \nu_j \delta_i X_{\tau}^{j} \\
&= \sum_{i,j}^n  \delta_i \delta_i ( X_{\tau}^{j} \, \nu_j)  = \Delta_M  \langle  X_{\tau}, \nu \rangle = 0,
\end{split}
\end{equation}
where $\Delta_M = \sum_{i=1}^n \delta_i \delta_i$ is the Laplacian on $M$. 
Together with \eqref{int_by_parts1}, \eqref{int_by_parts2} and \eqref{mixed_terms} the formula \eqref{all_terms} becomes
\begin{equation}
\label{all_terms2}
\begin{split}
P''(0) = & \int_{M} |D_{\tau} \langle X , \nu  \rangle|^2- |B_M|^2 \langle X, \nu \rangle^2 \, d \Ha^{n-1} \\
&-  \int_{M \cap \partial \Omega} B_{\partial \Omega}(X, X) \,d \Ha^{n-2}  - 2\int_{M \cap \partial \Omega}\langle X , \nu  \rangle   \langle D \nu \,  X_{\tau}  , \nu_{\Omega}\rangle \, d\Ha^{n-2}  \\
&+  \int_{M}H \left(  \langle X , \nu \rangle  \diver X  -     \diver_\tau (X_{\tau} \langle X , \nu \rangle )  \right)\, d \Ha^{n-1}  \\
&+ \int_{M} ( \diver_\tau X_{\tau})^2+  |(D_{\tau}X_{\tau})^T \nu|^2-  \trace   (D_{\tau} X_{\tau})^2 + H \langle Z_{\tau}, \nu \rangle  \, d\Ha^{n-1} .
\end{split}
\end{equation}

With the aim of understanding the last row in the previous formula, we observe that, as we noticed before, $X_{\tau}$ vanishes on $M \cap \partial \Omega$ and therefore the flow associated  by  $X_{\tau}$ leaves $M$ unchanged. The second variation (formula \eqref{E''}) of the perimeter with in the direction the vector field $X_{\tau}$ in then zero, in other words we have
\begin{equation}
\label{unchanged}
\int_{M} ( \diver_\tau X_{\tau})^2+  |(D_{\tau}X_{\tau})^T \nu|^2-  \trace   (D_{\tau} X_{\tau})^2 + H \langle   Z_{\tau}, \nu \rangle  \, d\Ha^{n-1} +  \int_{M \cap \partial \Omega}  \langle Z_{\tau}, \nu_{\Omega} \rangle  \, d\Ha^{n-2} = 0.
\end{equation}

Since $X_{\tau}$ vanishes on $M \cap \partial \Omega$, we have $D \langle X_{\tau}, \nu_{\Omega} \rangle \cdot X_{\tau} = 0$  on. This yields
\begin{equation}
\label{boundary1}
\int_{M \cap \partial \Omega}  \langle Z_{\tau}, \nu_{\Omega} \rangle  \, d\Ha^{n-2}=  \int_{M\cap \partial \Omega}  \langle D(X_{\tau}) X_{\tau}, \nu_{\Omega} \rangle \, d \Ha^{n-2} = - \int_{M \cap \partial \Omega} B_{\partial \Omega}(X_{\tau}, X_{\tau}) \,d \Ha^{n-2}.
\end{equation}
Moreover, since $\langle \nu, \nu_{\Omega}\rangle = 0$ on  $M \cap \partial \Omega$  we get $0= D \langle \nu, \nu_{\Omega} \rangle \cdot X_{\tau} =  \langle D \nu \,  X_{\tau}  , \nu_{\Omega}\rangle +  \langle D \nu_{\Omega} \,  X_{\tau}  , \nu\rangle$ and therefore
\begin{equation}
\label{boundary2}
\begin{split}
\int_{M \cap \partial \Omega}\langle X , \nu  \rangle   \langle D \nu \,  X_{\tau}  , \nu_{\Omega}\rangle \, d\Ha^{n-2} =  &  -\int_{M \cap \partial \Omega} \langle D \nu_{\Omega} \,  X_{\tau}  , X_{\nu}\rangle\, d\Ha^{n-2}  \\ 
= &  -\int_{M \cap \partial \Omega} B_{\partial \Omega} ( X_{\tau}  , X_{\nu})\, d\Ha^{n-2}.
\end{split}
\end{equation}
We combine \eqref{2VarNonLoc}  and \eqref{all_terms2}, together with  \eqref{unchanged}, \eqref{boundary1} and \eqref{boundary2} to write the second variation as
\begin{equation}
\label{almost_there}
\begin{split}
\frac{d^2 J(E_t)}{dt^2} \big|_{t=0} =& \int_{M} \left( \big| D_{\tau}\langle X, \nu_M \rangle \big|^2- | B_M |^2  \langle X, \nu_M \rangle^2 \right)\, d\Ha^{n-1}  -  \int_{M \cap \partial \Omega}  B_{\partial \Omega}(\nu,\nu) \langle X , \nu  \rangle^2  \,d \Ha^{n-2} \\
&+  \int_{M}H \left(  \langle X , \nu \rangle  \diver X  -     \diver_\tau (X_{\tau} \langle X , \nu \rangle )  \right)\, d \Ha^{n-1}  \\
&+ 8 \gamma \int_{M}  \int_{M}  G(x,y) \langle X(x), \nu \rangle   \langle X(y), \nu \rangle\, d\Ha^{n-1}(x)d\Ha^{n-1}(y)\\
&+  4 \gamma \int_{M}  \diver (v_E X)\,   \langle X, \nu \rangle \, d\Ha^{n-1}.
\end{split}
\end{equation}

Finally we integrate by parts to obtain
\[
\begin{split}
 \int_{M}  \diver (v_E X)\,   \langle X, \nu \rangle \, d\Ha^{n-1} = &  \int_{M}  \langle Dv_E , X\rangle   \langle X, \nu \rangle \, d\Ha^{n-1}+  \int_{M}  v_E \diver ( X)\,   \langle X, \nu \rangle \, d\Ha^{n-1} \\
= &  \int_{M}  \langle D_{\tau}v_E , X_{\tau}\rangle   \langle X, \nu \rangle \, d\Ha^{n-1} + \int_{M}  \langle \nabla v_E, \nu \rangle  \langle X, \nu \rangle^2 \, d\Ha^{n-1} \\
 & +  \int_{M}  v_E \diver ( X)\,   \langle X, \nu \rangle \, d\Ha^{n-1} \\
= & - \int_{M}  v_E  \diver ( X_{\tau} \langle X, \nu \rangle ) \, d\Ha^{n-1} + \int_{M} \langle \nabla v_E, \nu \rangle    \langle X, \nu \rangle^2 \, d\Ha^{n-1} \\
 & +    \int_{M}  v_E \diver ( X)\,   \langle X, \nu \rangle \, d\Ha^{n-1},
\end{split}
\]
where the last equality follows again from the fact that $X_{\tau}$ vanishes on $M \cap \partial \Omega$. The last equality, combined with \eqref{almost_there}, yields the result.
\end{proof}

\section{$W^{2,p}$-minimality}

In this section we prove that a regular  critical set $E$ with positive second variation is a strict local minimizer among sets which are regular and close to $E$ in a strong sense, namely in the $W^{2,p}$-topology, and which satisfy the orthogonality condition \eqref{orthogonality}. This result is stated in Proposition \ref{smooth} and can be interesting in itself. The idea is to take a competitor set $F$ near $E$ and to construct a volume preserving flow $\Phi$ such that   $\Phi(E,1) =F$. We can use then Proposition \ref{2ndVariation}  to estimate the second variation of $E_t = \Phi(E,t)$ at any time $t \in [0,1]$. Notice that by the assumption on $E$ we know that the second derivative of $t \mapsto J(E_t)$ at $t=0$ is strictly positive. We then use the fact that $E_t$ are close to $E$ in $W^{2,p}$ topology to deduce that the function  $t \mapsto J(E_t)$  is in fact strictly convex. 

In the following we say that $J$ has \emph{positive second variation} at the set $E$ if 
\[
\partial^2 J(E)[\varphi] \geq 0 \qquad \text{for every }\, \varphi \in H^1(M) \,\, \text{with }\,  \int_{M}\varphi \, d \Ha^{n-1}= 0.
\]  
where we recall that $\partial^2 J(E)[\varphi]$ is defined by \eqref{quadratic}. We begin with a simple compactness argument whose proof is exactly the same as   \cite[Lemma 3.6]{AFM} and  will  therefore be omitted. 
\begin{lemma} 
\label{easy_lemma}
Suppose that $J$ has positive second variation at the critical set $E$. Then there exists $c_0>0$ such that 
\[
\partial^2 J(E)[\varphi] \geq  c_0 \|\varphi \|_{H^1(M)}^2
\] 
 for every $\varphi \in H^1(M)$ with $\int_{M}\varphi \, d \Ha^{n-1}= 0$. 
\end{lemma}
We define the $W^{2,p}$ and $C^2$ distance between regular sets $E,F\subset \Omega$ as
\begin{equation*} 
\| E,F \|_{W^{2,p}} := \inf \{ \| \Psi - Id \|_{W^{2,p}(\Omega)} \mid  \Psi : E \to F, \,\, C^{2}-\text{diffeomorphism } \}
\end{equation*}
\begin{equation*}
\| E,F \|_{C^{2}} := \inf \{ \| \Psi - Id \|_{C^{2}(\Omega)} \mid  \Psi : E \to F, \,\, C^{2}-\text{diffeomorphism } \}
\end{equation*}
The main result of the section is the following:
\begin{proposition} 
\label{smooth}
Let $E$ be as in  Theorem \ref{mainthm} and $p>n$. There exist $\delta>0 $ and a constant $c_1 >0$ such that for any $F \subset \Omega$ with $|F|= |E|$, satisfying the orthogonality condition \eqref{orthogonality} and such that $\| E,F \|_{W^{2,p}} \leq \delta$, it holds
\[
J(F) \geq J(E) + c_1|F \triangle E|^2.
\]
\end{proposition}

We start proving some technical lemmata and then we give the proof of  Proposition \ref{smooth} at the end of the section. A crucial point is theLemma \ref{the_flow} where we construct a vector field $X$ and a flow $\Phi$ which deforms the set $E$ into a given regular set $F$ sufficiently close to $E$. The difficulty lies in the fact that the flow needs to satisfy both the orthogonality and the volume constraint. Therefore we have to construct it carefully near the boundary $\partial \Omega$ in order to preserve the orthogonality condition. 

Let us recall some well know facts on the distance function from a regular set.  Since $\Omega$ is $C^{4, \alpha}$-regular, the distance function $d_{\Omega}(x)= \inf_{y \in \partial \Omega} |x-y|$ is $C^{4, \alpha}$-regular in a neighborhood of $\partial \Omega$. By a neighborhood of $\partial \Omega$ we mean a connected set $V\subset \overline\Omega$ which contains $\partial \Omega$ and it is relatively open with respect to $\overline{\Omega}$. We may define the projection $\Pi: V \to \partial \Omega$ as $\Pi(x)= y_x$, where $y_x \in \partial \Omega$ is the unique point for which $d_{\Omega}(x)= |x- y_x|$. Every point $x \in V$ can therefore be written  as  
\begin{equation}
\label{proj_omega}
x = \Pi(x) - d_{\Omega}(x) \nu_{\Omega}(\Pi(x)),
\end{equation}
where $\nu_{\Omega}(y)$ is the outer normal of $\Omega$ at $y \in \partial \Omega$. We also remark that we may naturally define $D\Pi(y)$ for every $y \in \partial \Omega$ and the Kernel of $D\Pi(y)$ is spanned by $\nu_{\Omega}(y)$
\begin{equation}
\label{spanned}
\text{Ker}\, (D\Pi(y)) = \text{span} \{\nu_{\Omega}(y)\} \qquad y \in \partial \Omega.
\end{equation}
The orthogonality condition \eqref{orthogonality} for a set $E$ is equivalent to the fact that for every $x \in  M \cap \partial \Omega$ the normal vector to $\partial\Omega$ at $x$ belongs to the tangent plane of $M$ at $x$, i.e., 
\begin{equation}
\label{orthogonality2}
\nu_{\Omega}(x) \in \tang_x(M).
\end{equation}

Finally we remark that we may consider $\partial \Omega$ itself  as a $n-1$ dimensional manifold and define a natural distance on $\partial \Omega$ by
\begin{equation} \label{dist_omega}
\dist_{\partial \Omega}(x, y) := \inf\left\{ \int_{0}^1|\xi'(s)|\, ds \, : \, \xi \, \text{ is a path on $\partial \Omega$ such that $\xi(0)= x$ and  $\xi(1)= y$}\right\}. 
\end{equation}

\begin{lemma}
\label{the_flow}
Let $E$ be as in   Theorem \ref{mainthm} and  let $F\subset \Omega$  be a $C^2$-regular set which satisfies the volume constraint $|F|= |E|$ and the orthogonality condition \eqref{orthogonality}. There exists $\delta >0$ such that, if $\| E,F \|_{W^{2,p}} \leq \delta$, with $p >n$, then we can find a $C^2$-regular vector field $X $, with  $\| X\|_{W^{2,p}} \leq C\delta $, that satisfies the tangential condition \eqref{tangent.cond} and such that  the associated flow $\Phi$, defined by \eqref{flow}, has the following properties:
\begin{enumerate}
\item For any $t \in [0,1]$ the set $E_t = \Phi(E, t)$ satisfies the volume constraint $|E_t|= |E|$ and the orthogonality condition \eqref{orthogonality}
\item The flow $\Phi$ takes $E$ to $F$, i.e., $E_1 = \Phi(E,1) = F$.
\end{enumerate}
\end{lemma}

\begin{proof}
The proof is rather technical and long and it is therefore divided into three steps. 

\textbf{Step 1:} We begin by constructing a vector field, which gives us the trajectories of the final flow. It will be enough to construct the vector field  in a neighbourhood of the relative boundary $M = \overline{\partial E \cap \Omega}$, which we denote by $U$. By a neighbourhood of $M$ we mean a connected set $U$ which contains $M$ and is relatively open with respect to $\overline{\Omega}$.

First of all, by Proposition \ref{int_regularity}, $M$ is $C^{4, \alpha}$ regular and therefore the signed  distance function from $M$  in $\Omega$
\[
d_M(x):= \left\{
\begin{aligned}
 &\inf_{y \in M} |x-y| &&\text{for } x \in \Omega \setminus E \\
 - &\inf_{y \in M} |x-y| &&\text{for } x \in \Omega \cap E\\
\end{aligned}
\right.
\]
is in  $C^{4, \alpha}(U)$. The gradient field $\nabla d_M$ will define the trajectories away from the boundary of $\Omega$.

As we already pointed out, we have to construct the trajectories carefully near the boundary $\partial \Omega$ in order to be able to verify the orthogonality condition in the forthcoming step. We begin by noting that the $n-2$ dimensional boundary $M \cap \partial \Omega$ is $C^{4, \alpha}$ regular. The signed distance function of  $M \cap \partial \Omega$ on $ \partial \Omega$
\[
d_{M\cap \partial \Omega}(x):= \left\{
\begin{aligned}
 &\inf_{y \in M \cap \partial \Omega} \dist_{\partial \Omega}(x,y) &&\text{for } x \in \partial \Omega \setminus E \\
 - &\inf_{y \in M \cap \partial \Omega}  \dist_{\partial \Omega}(x,y) &&\text{for } x \in \partial \Omega \cap E\\
\end{aligned}
\right.
\]
is therefore $C^{4, \alpha}$ regular near $M\cap \partial \Omega$. Here  $\dist_{\partial \Omega}(x,y)$ is defined in \eqref{dist_omega}. We will define the vector field $Z$ on the boundary $\partial \Omega$ as
\begin{equation} \label{def_Z_b}
Z(x) = \nabla_{\tau} d_{M\cap \partial \Omega}(x) \qquad \text{for }\, x \in \partial \Omega \cap U.
\end{equation}
Here   $\nabla_{\tau} $ denotes the tangential gradient on   $\partial \Omega$. We will extend $Z$ to a neighbourhood of $\partial \Omega \cap U$, which we will denote by $V$, such that $Z \in C^{2}(V, \R^n)$. The trajectories of the final flow  are then given by the vector field
\begin{equation}\label{Y'}
Y'(x)= \zeta(x)  \nabla d_M(x) + (1- \zeta(x)) Z(x), 
\end{equation}
where $\zeta \in C_0^{\infty}(U)$ is a cut-off function such that $\zeta \equiv 1$ outside $V$. Notice that then $Y' \in C^{2}(U, \R^n)$ and \eqref{def_Z_b} implies that $Y'$ satisfies the tangent condition \eqref{tangent.cond}. Since $Y'$ defines the trajectories for the final flow $X$, it will then also satisfy \eqref{tangent.cond}.

We will extend $Z$, defined on $\partial \Omega$ as \eqref{def_Z_b},  to $V$ such that it satisfies 
\begin{equation} \label{Z_div=0}
\diver Z = 0 \qquad \text{in }\, V
\end{equation}
and it has a property which we call ''the projection property''. To describe this, let us assume that we have extended $Z$ to $V$, and let $\Phi_Z$ be the associated flow 
\[
\frac{\partial}{\partial t} \Phi_Z(x,t) = Z( \Phi_Z(x,t)) \quad \text{with }\,\Phi_Z(x, 0) =x 
\]
defined in $V$. Let $\Pi$ be the projection in $V$ onto $\partial \Omega$ defined prior to \eqref{proj_omega}. We construct $Z$ such that the flow $\Phi_Z$ satisfies 
\begin{equation} \label{projection_Z}
\Pi\left( \Phi_Z(x, t) \right) =  \Phi_Z( \Pi(x), t)
\end{equation}
for every $x \in V$ and $t \in \R$ for which $\Phi_Z(x, t) \in V$. Roughly speaking \eqref{projection_Z} means that any point $y \in V$ will travel  side-by-side with its projection point $\Pi(y)$. In other word, for every $x \in \partial \Omega$ and $h >0$ it holds
\[
 \Phi_Z( x- h \nu_{\Omega}(x), t)  = \Phi_Z(x, t) -  l_h \nu_{\Omega}( \Phi_Z(x, t) )
\]
for some $l_h >0$. We call this the projection property of $Z$ (cf. Figure \ref{fig-01}).

  \begin{figure} 
  \label{fig-01}
   \centering
   \def\svgwidth{9cm} 
   \vspace{-2cm}
   \small{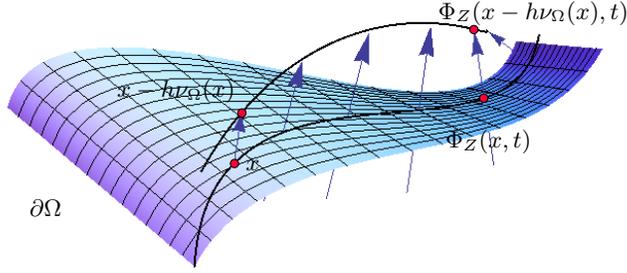}
   \vspace{-2cm}
   \caption{The projection property on a trajectory}\label{fig-01}
 \end{figure}
 
First we remark that if $Z$ is an extension of  \eqref{def_Z_b} and satisfies \eqref{Z_div=0} and \eqref{projection_Z}, then it will be uniquely defined and  we may construct it  locally. We first make sure that $Z$  will satisfy \eqref{projection_Z}. In order to do this, it is convenient to flatten the boundary $\partial \Omega$. 

 Let us  fix $x_0 \in M \cap \partial \Omega$, which we may assume to be the origin. We write  $x \in \R^n$ as $x = (x',x_n)$ where  $x' = (x_1, \dots, x_{n-1}) \in \R^{n-1}$ and denote the projection $\tilde{\Pi}: \R^n \to \R^n$ onto the plane $\{x_n = 0 \}$ by   
\[
\tilde{\Pi}(x)= (x', 0). 
\] 
We may locally write, up to a translation and rotation, $\Omega = \{ x \mid x_n > g(x')\}$ for a $C^{4,\alpha}$ function $g : D \subset \R^{n-1} \to \R$ with $g(0)=0$ and $\nabla g(0)= 0$.  We define a $C^{4,\alpha}$-diffeomorphism $\Psi$ in a neighbourhood of  $x_0$, say $V' \subset V$, which flattens the boundary. Since we may write  every point $x \in \partial \Omega$ as $x = (x', g(x'))$, we define the diffeomorphism on the boundary as  $\Psi(x)= \tilde{\Pi}(x)$. By \eqref{proj_omega}, every $x \in V'$ can be expressed as $x=  \Pi(x) - d_{\Omega}(x) \nu_{\Omega}(\Pi(x))$. The diffeomorphism $\Psi: V' \to \Psi(V')$ is then defined as 
\[
\Psi(x)= \tilde{\Pi}(\Pi(x)) + d_{\Omega}(x)e_n.
\] 
Notice that $e_n$ is the inner normal of the new domain $\Psi(V')$ on the flat boundary. The map $\Psi$ has the nice property that for every $x \in V'$ and every $y \in \Psi(V')$ it holds
\begin{equation} \label{proj_diff}
\tilde{\Pi}(\Psi(x)) = \Psi(\Pi(x)) \qquad \text{and} \qquad \Pi(\Psi^{-1}(y)) = \Psi^{-1}(\tilde{\Pi}(y)).
\end{equation}

Notice that the flow $\Phi_Z$ is already defined on $\partial \Omega$ by \eqref{def_Z_b}. We map this flow to the new coordinates by  
\[
\tilde{\Phi}_Z(y,t) :=\Psi(\Phi_Z(\Psi^{-1}(y),t))
\]
and denote the associated vector field by $\tilde{Z}$. We may write $\tilde{Z}$ explicitly as
\begin{equation} \label{tildeZ}
\tilde{Z}(y)= D\Psi(\Psi^{-1}(y)) Z(\Psi^{-1}(y))
\end{equation}
on the flat boundary of $\Psi(V')$. Since $\tilde{\Phi}_Z$ does not exit the flat boundary $\partial (\Psi(V')) \subset \{  y_n= 0\}$,  its vertical velocity is zero, i.e.,  $\langle \tilde{Z}, e_n \rangle = 0$. In other words, $\tilde{\Pi}(\tilde{Z})= \tilde{Z}$, where $\tilde{\Pi}$ is the projection onto  the plane $\{  y_n= 0\}$. We extend $\tilde{Z}$ to $\Psi(V')$ as 
\begin{equation}
\label{extension_Z_1}
\tilde{Z}(y):= \tilde{Z}(\tilde{\Pi}(y)) + f(y) e_n \qquad \text{for }\, y \in \Psi(V'),
\end{equation}
where $f \in C^{2}(\Psi(V')) $ is some function with boundary values $f(y)=0$  on $y \in \{  y_n= 0\}$, which will be chosen later. 

Let us  denote, with a slight abuse of notation, the associated flow by $\tilde{\Phi}_Z$. The construction implies that
\[
\tilde{\Pi}(\tilde{Z}(y)) = \tilde{Z}(\tilde{\Pi}(y)) \qquad \text{for }\, y \in \Psi(V').
\]
This means that at every point $y \in \Psi(V')$ the flow has the same horizontal velocity as at its projection point $\tilde{\Pi}(y)$. Therefore $\tilde{\Phi}_Z$ satisfies
\begin{equation}
\label{project_Psi}
\tilde{\Pi}\left( \tilde{\Phi}_Z(y, t) \right) =  \tilde{\Phi}_Z( \tilde{\Pi}(y), t).
\end{equation}

We map the flow $\tilde{\Phi}_Z$ back to $V'$ by 
\begin{equation} 
\label{defin_Phi_Z}
\Phi_Z(x,t) := \Psi^{-1}(\tilde{\Phi}_Z(\Psi(x),t)),
\end{equation}
for every $x \in V'$ and $t \in \R$ for which $\Phi_Z(x,t) \in V'$. It follows from \eqref{proj_diff}, \eqref{project_Psi} and \eqref{defin_Phi_Z} that
\[
\begin{split}
\Pi\left( \Phi_Z(x, t) \right)  &=  \Pi\left(\Psi^{-1}(\tilde{\Phi}_Z(\Psi(x),t))  \right)\\
&=  \Psi^{-1}\left(\tilde{\Pi}(\tilde{\Phi}_Z(\Psi(x),t)) \right)\\
&= \Psi^{-1}\left( \tilde{\Phi}_Z( \tilde{\Pi}(\Psi(x)), t) \right)\\
&=  \Psi^{-1}\left( \tilde{\Phi}_Z( \Psi(\Pi(x)), t) \right)\\
&= \Phi_Z\left(\Pi(x),t\right).
\end{split}
\]
Hence, we have the projection property \eqref{projection_Z}.

Let $Z$ be the vector field  associated  to $\Phi_Z$. We may use \eqref{tildeZ},  \eqref{extension_Z_1}, \eqref{proj_diff} and \eqref{defin_Phi_Z} in order to write $Z$ explicitly as
\begin{equation} \label{Z_explicit}
\begin{split}
 Z(x) &= D\Psi^{-1}(\Psi(x)) \tilde{Z}(\Psi(x)) \\
&=  D\Psi^{-1}(\Psi(x)) \left( \tilde{Z}(\tilde{\Pi}(\Psi(x))) + f(\Psi(x)) e_n \right) \\
&=   D\Psi^{-1}(\Psi(x)) \left( \tilde{Z}(\Psi(\Pi(x))) + f(\Psi(x)) e_n \right)\\
&=   D\Psi^{-1}(\Psi(x)) D\Psi( \Pi(x)) Z(\Pi(x)) + f(\Psi(x))    D\Psi^{-1}(\Psi(x)) e_n .
\end{split}
\end{equation}
We remark that, since we impose the boundary condition $f(y)= 0$  for $y_n = 0$, we have  $f(\Psi(\Pi(x))) = f(\tilde{\Pi}(\Psi(x))= 0$. Therefore it follows from \eqref{Z_explicit} that $Z$ really is an extension of the vector field we defined on the boundary  $\partial \Omega$ in \eqref{def_Z_b}.

Finally we choose $f$, which defines the vertical velocity of $\tilde{Z}$ in \eqref{extension_Z_1}, such that $Z$ satisfies \eqref{Z_div=0}.  Let us write the equation  $\diver Z = 0$ in a form
\[
\langle  {\bf{a}}(x), \nabla f(\Psi(x)) \rangle + b(x)f(\Psi(x))  + c(x)= 0 \qquad x \in V',
\]
where, by \eqref{Z_explicit}, the coefficients are
\[
\begin{aligned}
{\bf{a}}(x) &=  D \Psi(x) D\Psi^{-1}(\Psi(x)) e_n= e_n, \\
b(x) &= \diver\left(    D\Psi^{-1}(\Psi(x)) e_n \right) ,\\
c(x) &= \diver\left( D\Psi^{-1}(\Psi(x)) D\Psi( \Pi(x)) Z(\Pi(x)) \right).\\
\end{aligned}
\]
We remark that since $\Psi$ is a $C^{4,\alpha}$-diffeomophism and $x \mapsto Z(\Pi(x))$ is $C^{3,\alpha}$ regular  (recall  \eqref{def_Z_b}) then the above coefficients are $C^{2,\alpha}$ regular. Going back to the domain $\Psi(V')$ with a flat boundary, we have to solve the PDE
\begin{equation}
\label{thePDE_Z}
\langle \nabla f(y), e_n \rangle + \tilde{b}(y)f(y)  + \tilde{c}(y)= 0 \qquad y \in \Psi(V'),
\end{equation}
with the boundary condition $f(y)= 0$ on $ \Psi(V')\cap \{ y_n = 0\}$. Here  $\tilde{b}(y) = b(\Psi^{-1}(y))$ and  $\tilde{c}(y) = c(\Psi^{-1}(y))$. We may solve this equation by using  the method of characteristics and thus we have $Z$, which is $C^{2}$ regular.

We define the primitive vector field $Y'$ by \eqref{Y'}. For the distance function of $M$ it holds  $\nabla d_M(x)= \nu_M(x)$ for $x \in M$. Moreover, since $M$ satisfies the orthogonality condition \eqref{orthogonality}, \eqref{def_Z_b}  implies that for $x \in M \cap \partial \Omega$ it holds $Z(x) = \nabla_{\tau} d_{M\cap \partial \Omega}(x) = \nu_M(x)$. By the regularity of $Z$ and $M$ we conclude that we have 
\begin{equation}\label{reg_Y'}
| Y'(x) - \nu_M(x) | \leq \frac{1}{16} \qquad x \in  M,
\end{equation}
by  choosing the cut-off function  such that $\zeta \neq 1$  only in a narrow  neighbourhood of $\partial \Omega$.

\bigskip

\textbf{Step 2:} In this short step we change $Y'$ to $Y \in C^{2}(U, \R^n)$, which has the same trajectories and satisfies
\begin{equation}\label{div_free_Y}
\diver(Y(x)) = 0 \quad x \in U \qquad \text{and} \qquad Y(x)= Y'(x) \quad \text{on }\, M,
\end{equation}
where $U$ is a neighbourhood of $M$. We impose that $Y$ is of the form $Y(x)= \eta(x)Y'(x)$, where $\eta$ is a function. In order to have $\diver Y = 0$ and $Y(x)= Y'(x)$ on $M$, the function $\eta$ has to satisfy 
\begin{equation}\label{div_free}
\langle Y'(x), \nabla \eta(x) \rangle + \eta(x)\diver Y'(x) = 0 \qquad x \in  U
\end{equation}
with the boundary condition $\eta(x) = 1$ for $x \in M$.  By \eqref{Y'}, $Y'$ is of the form $Y'= (1- \zeta)Z + \zeta \nabla d_M$, where $Z$ is  $C^{2}$-regular and   $\nabla d_M $ and $\zeta$ are  $C^{3}$-regular. However, since $\diver Z = 0$, the equation \eqref{div_free} is $C^2$-regular and we may solve it by using the method of characteristics. Moreover, since $Y' = Z$ and $\diver Z = 0$ in a neighbourhood of $\partial \Omega$, say $V$, we have that 
\begin{equation}\label{near_bdry}
Y(x) = Z(x) \qquad x \in  V.
\end{equation}

We denote the flow associated to $Y$ by $\Phi_Y$. It is clear that $\Phi_Y$ is defined in $U$ for a short  time interval $(-t_0, t_0)$. Moreover, since  $Y \in C^{2}(U, \R^n)$, $\Phi_Y$ is also $C^{2}$-regular. Moreover, from  \eqref{reg_Y'} and  \eqref{div_free_Y} it follows that
\begin{equation}\label{reg_Y}
| Y(x) - \nu_M(x) | \leq \frac{1}{16}\qquad x \in  M.
\end{equation}
We may also assume 
\begin{equation}\label{bound_Y}
\| Y\|_{L^{\infty}(U)} \leq 2
\end{equation}
by possibly choosing a smaller neighbourhood $U$.

 \bigskip

\textbf{Step 3:} We finally construct the vector field $X$ such that the associated flow, denoted by $\Phi$, satisfies $\Phi(E, 1) = F$, and that for every $t \in [0,1]$ the set $E_t:= \Phi(E, t)$ satisfies the orthogonality condition \eqref{orthogonality} and the volume constraint $|E_t|= |E|$. 

In a neighbourhood $U$  of $M$ the map $\Phi_Y |_{M\times (-t_0, t_0)} \to U$ is  a $C^{2}$-diffeomorphism and every point $y \in U$ can be written uniquely  as $y = \Phi_Y(x, t)$ for some $x \in M$ and for some  $t \in (-t_0, t_0)$. We may therefore define implicitly the projections $\pi: U  \to M$  and $T :U \to (-t_0, t_0)$ such that 
\begin{equation} \label{definition_T_pi}
\Phi_Y(\pi(y), T(y))= y.
\end{equation}
Since  $\Phi_Y$ is $C^{2}$ regular, also $\pi$ and $T$ are  $C^{2}$ regular.

By the assumption on $F$, there is a $C^2$-diffeomorphism $\Psi_F: E \to F$ with $\|\Psi_F- Id\|_{W^{2,p}} < \delta$. Without further mention we always assume $\delta >0$ to be small.  We define the map $S: M \to M$ as 
\[
S(x) := \pi(\Psi_F(x)).
\]
The tangential differential of $S$ is
\[
D_{\tau}S(x) = D_{\tau} \pi(\Psi_F(x)) \, D_{\tau}\Psi_F(x) \qquad \text{on }\, x \in M.
\]
From the regularity of  $\pi$ and from the fact that $\pi(x)=x$ for $x \in M$ we conclude that $|D_{\tau} \pi(\Psi_F(x))|\geq c$. Moreover, since $\|D_{\tau}\Psi_F(x)- I\|_{L^{\infty}} \leq \delta$ we  conclude that $S$ is a $C^2$-diffeomorphism and the bound $\|\Psi_F\|_{W^{2,p}} \leq C$ implies
\begin{equation} \label{regularity_S}
\|S^{-1}\|_{W^{2,p}(M)} \leq C. 
\end{equation}

We define further a projection onto the relative boundary of $F$, $M_F := \overline{\partial F \cap \Omega}$, $\pi_F:U \to M_F$ as
\[
\pi_F(y) := \Psi_F(S^{-1}(\pi(y))),
\]
which labels for every point $y \in U$ a unique point $z$ on $M_F$ on the same trajectory. Trivially the map $\pi_F$ is constant along the trajectories of $\Phi_Y$ and the value 
\begin{equation}
\label{def_T_F}
T_F(y):= T (\pi_F(y))
\end{equation}
denotes the time needed from $M$ to $M_F$ along the trajectory passing through $y$. The definition of $\pi_F$ and the estimates  $\|\Psi_F- Id\|_{W^{2,p}} < \delta$ and \eqref{regularity_S} imply 
\[
\|\pi_F-  S^{-1} \circ \pi \|_{W^{2,p}} < C \delta.
\] 
Since $S^{-1}(\pi(y)) \in M$ for every $y \in U$, it holds, by the definition of $T$, that $T(S^{-1}(\pi(y))) =0$ for $y \in U$. Therefore the above estimate yields
\begin{equation}
\label{regularity_T_F}
 \|T_F\|_{W^{2,p}} = \|  T \circ \pi_F -  T \circ S^{-1} \circ \pi\|_{W^{2,p}} \leq C \delta.
\end{equation}

We define the vector field $X$ by
\begin{equation} \label{def_X}
X(y) = T_F(y) Y(y).
\end{equation}
 Since $Y$ is $C^2$ regular, the estimate \eqref{regularity_T_F} yields $ \|X\|_{W^{2,p}} \leq C \delta$.  We denote  the flow  associated to $X$  by $\Phi$. We may write $\Phi$ explicitly as
\begin{equation} 
\label{def_Phi_Y}
\Phi(x, t)= \Phi_Y(x, T_F(x) \, t) \qquad \text{for }\, x \in M.
\end{equation}
We denote  $E_t= \Phi(E, t)$ and  $M_{t} = \overline{\partial E_{t} \cap \Omega}$. Since $T_F(x)$ is the time that the flow $\Phi_Y$ needs to get  from $x\in M$ to $M_F$, it follows directly from  \eqref{def_Phi_Y}  that $ M_F = \{ \Phi(x, 1) \mid x \in M\}$. Hence $E_1= F$. 

Since $T_F$ is constant along the trajectories of $\Phi_Y$, \eqref{div_free_Y} implies 
\begin{equation} 
\label{def_Phi_X}
\diver X = \diver (T_F Y)= \langle \nabla T_F, Y \rangle + T_F \diver Y = 0.
\end{equation}
We may  calculate as in  \eqref{2nd_vol} 
\[
\frac{d^2}{dt^2} |E_t| = \int_{M_t} \diver (X)  \langle X, \nu_M \rangle \, d \Ha^{n-1} = 0
\]
by \eqref{def_Phi_X}. Therefore $t \mapsto |E_t|$ is an affine function. Since $|E| = |F|$, it has to be constant. Hence, the flow  satifies  the volume constraint $|E_t|= |E|$ for $t \in [0,1]$.

We have yet to make sure that at any time $t \in (0,1)$ the set $E_t$  satisfies the orthogonality condition. Before that we remark that \eqref{near_bdry} implies that $\Phi_Y = \Phi_Z$ in a neighbourhood of $\partial \Omega$ which we denote by $V$. We show first that for every $z \in \partial \Omega$ and every $s \in (-t_0, t_0)$ it holds 
\begin{equation} \label{normal_evolv2}
 D\Phi_Z(z, s) \nu_{\Omega}(z) = \lambda\nu_{\Omega}(\Phi_Z(z, s)) 
\end{equation}
for some $\lambda = \lambda(z,s)  >0$.

Indeed, the projection property \eqref{projection_Z} yields
\[
\Pi\left( \Phi_Z(z - h \nu_{\Omega}(z), s) \right) = \Phi_Z\left(\Pi(z - h \nu_{\Omega}(z)), s\right) =  \Phi_Z(z,s)
\]
for every small $h>0$. This implies
\[
0 = \lim_{h \to 0}\frac{1}{h} \left(\Pi\left( \Phi_Z(z - h \nu_{\Omega}(z), s) \right) - \Phi_Z(z,s) \right) =- D\Pi (\Phi_Z(z,s)) D\Phi_Z(z,s) \nu_{\Omega}(z).
\]
In other words, $ D\Phi_Z(z, s) \nu_{\Omega}(z) \in \text{Ker}( D\Pi (\Phi_Z(z,s)))$, which implies \eqref{normal_evolv2} by \eqref{spanned}.

Let us now prove the orthogonality condition. We fix $t\in (0,1)$ and   $y  \in M_{t} \cap \partial \Omega$. We will show that 
 \begin{equation} \label{normal_evolv}
\nu_{\Omega}(y) \in \tang_{y}(M_{t}),
\end{equation}
which  implies the orthogonality as we noted in \eqref{orthogonality2}. Since $y  \in M_{t} \cap \partial \Omega$, there exists a unique $x \in M \cap \partial \Omega$ such that $y = \Phi(x , t) = \Phi_Z(x, T_F(x)t)$.

To show \eqref{normal_evolv}, we notice that since $M$ satisfies the orthogonality condition \eqref{orthogonality}, it holds $ \nu_{\Omega}(x) \in \tang_x(M)$ by \eqref{orthogonality2}. Therefore, since $\Phi(\cdot, t): M \to M_t$ is a diffeomorphims, we have 
\begin{equation} \label{tang_trivial}
 D\Phi(x, t) \nu_{\Omega}(x)   \in \tang_y(M_t).
\end{equation}
We use \eqref{def_Phi_Y} to calculate
\[
 D\Phi(x, t) \nu_{\Omega}(x)  = D\Phi_Z(x, T_F(x)t) \nu_{\Omega}(x) + t \frac{ \partial \Phi_Z}{\partial t} (x, T_F(x)t) \langle \nabla T_F(x),   \nu_{\Omega}(x)\rangle.
\]
Notice that \eqref{normal_evolv2} implies 
\[
 D\Phi_Z(x, T_F(x)t) \nu_{\Omega}(x)  = \lambda \nu_{\Omega}(y) 
\]
for some $\lambda >0$.  Hence, by  \eqref{tang_trivial} we need yet to to show 
\begin{equation} \label{ort_neu}
\langle \nabla T_F(x),   \nu_{\Omega}(x)\rangle = 0 
\end{equation}
to conclude \eqref{normal_evolv}.

We recall that $T_F(x)= T(\pi_F(x))$, where $\pi_F$ is a map which labels for every $x' \in U$ a unique point $y' \in M_F$ on the same trajectory. We may therefore write $\pi_F(x) = \Phi_Z(x, s)$ and $\pi_F(x- h \nu_{\Omega}(x)) = \Phi_Z(x- h \nu_{\Omega}(x) ,s_h)$ for some $s, s_h \in   (-t_0, t_0)$. We also denote $\tilde{y}= \pi_F(x)$ and 
\[
\lim_{h \to 0} \frac{s_h - s}{h} = \mu.
\] 
We  have
\begin{equation} \label{ugly_calc}
\begin{split}
D\pi_F(x) \nu_{\Omega}(x) &= \lim_{h \to 0} \frac{1}{h} \left( \pi_F(x) - \pi_F(x- h \nu_{\Omega(x)})  \right) \\
&=  \lim_{h \to 0} \frac{1}{h} \left( \Phi_Z(x,s) - \Phi_Z(x-  h \nu_{\Omega}(x) ,s_h) \right) \\
&=  D\Phi_Z(x, s) \nu_{\Omega}(x) - \mu \frac{\partial \Phi_Z}{\partial t}(x,s)\\
&= \tilde{\lambda} \nu_{\Omega}(\tilde{y}) - \mu Z(\tilde{y}), 
\end{split}
\end{equation}
where the last equality follows from \eqref{normal_evolv2} and $\frac{\partial \Phi_Z}{\partial t}(x,s) = Z(\Phi_Z(x,s) )$. Since $\pi_F(x) \in M_F$, we have that $D\pi_F(x) \nu_{\Omega}(x) \in \tang_{\tilde{y}}(M_F)$. Moreover, since $M_F$ satisfies the orthogonality condition \eqref{orthogonality2}, it holds 
\[
 \nu_{\Omega}(\tilde{y}) \in \tang_{\tilde{y}}(M_F).
\]
However, since $F$ is close to $E$ in $W^{2,p}$-sense, we have the estimate $|\nu_M(x) - \nu_{M_F}(\pi_F(x))| \leq C\delta $  (see \eqref{almost_id} in Lemma \ref{check_reg}). Since $Z = Y$ in $V$,  the estimate \eqref{reg_Y} and the $C^2$-regularity of $Z$ then imply 
\[
\langle  Z(\tilde{y}), \nu_{M_F}(\tilde{y})\rangle >0,
\] 
when $\delta$ is small. In particular, $ Z(\tilde{y}) \notin \tang_{\tilde{y}}(M_F)$. Therefore, by \eqref{ugly_calc}, we must have $\mu = 0$, which in turns implies 
\begin{equation} \label{nice_est}
D\pi_F(x) \nu_{\Omega}(x) =  \tilde{\lambda}  \nu_{\Omega}(\tilde{y})
\end{equation}
for some $\tilde{\lambda} >0$.

Since $T_F= T \circ \pi_F$, \eqref{ort_neu} can be written as
\[
\langle \nabla T(\pi_F(x)),  D \pi_F(x) \nu_{\Omega}(x)\rangle = 0 .
\]
By \eqref{nice_est} this is equivalent to
\begin{equation} \label{only_to_show}
\langle \nabla T(\tilde{y}),  \nu_{\Omega}(\tilde{y}) \rangle = 0,
\end{equation}
for $\tilde{y} = \pi_F(x)$. 

We recall that the definition of $T :U \to \R$  in \eqref{definition_T_pi} implies $T(\Phi_Y(x', t')) = t'$ for every $x' \in U$. Since $\Phi_Y = \Phi_Z$ in $V$, we have that
\[
T(\Phi_Z(x', s)) = s \qquad \text{for every }\, x' \in M \cap V.
\] 
We recall that $\tilde{y} = \Phi_Z(x, s)$ for $x \in M \cap \partial \Omega$. Therefore the above equality implies
\[
\langle \nabla T(\tilde{y}), D\Phi_Z(x, s) \tau \rangle = 0 \qquad \text{for every }\, \tau \in \tang_x(M).
\]
 Since $M$ satisfies the ortogonality condition, it holds $\nu_{\Omega}(x) \in \tang_x(M)$. Hence,
\[
\langle \nabla T(\tilde{y}), D\Phi_Z(x, s) \nu_{\Omega}(x)  \rangle = 0.
\]
We then conclude from  \eqref{normal_evolv2} that
\[
\langle \nabla T(\tilde{y}), \nu_{\Omega}(\tilde{y})  \rangle = 0.
\]
This proves \eqref{only_to_show} and shows that $M_t$ satisfies the orthogonality condition \eqref{orthogonality}, and the lemma is finally proved.
\end{proof}

In the next lemma we study the regularity properties of the flow $\Phi$ and the vector field $X$ constructed in the previous lemma. 

\begin{lemma}
\label{check_reg}
Let $E$, $F$ be as in Lemma \ref{the_flow}, let $X$ be the vector field defined by \eqref{def_X}, $\Phi$ the associated flow and 
set $M_t := \overline{\partial E_t \cap \Omega}$. The flow $\Phi$ satisfies
\begin{equation} \label{almost_id}
\| \Phi(\cdot, t)- Id\|_{W^{2,p}(M)}  \leq C \delta,
\end{equation}
for some constant $C>0$ and $\delta$ given by Lemma \ref{the_flow}.
The vector field $X$ satisfies 
\begin{equation} \label{symm_diff}
 |F \triangle E| \leq C \int_M |\langle X, \nu_M \rangle| \, d \Ha^{n-1} 
\end{equation} 
and
\begin{equation} \label{nasty_estimate}
  \|X_{\tau_{t}}\|_{H^1(M_t)} \leq C   \| \langle X, \nu_{M_t} \rangle\|_{H^1(M_t)}
\end{equation} 
for some constant $C$, where $\nu_{M_t}$ is the normal vector of $M_t$ and $X_{\tau_{t}} = X - \langle X, \nu_{M_t} \rangle \nu_{M_t}$. 
\end{lemma}

\begin{proof}
We begin by proving \eqref{almost_id}. By definition \eqref{flow} of $\Phi$, we can write, for $x \in M$ and $t \in [0,1]$, 
\[
\Phi(x,t) - x = \int_0^t X( \Phi(x, s) ) \, ds,
\]
then, recalling that  $X$ satisfies the estimate $\|X\|_{W^{2,p}}  \leq C\|T_F\|_{W^{2,p}}\leq C \delta$, we easily get $|\Phi(x,t) - x| \leq C \delta$. 

By differentiating \eqref{flow}, we obtain the equations
\[
\frac{\partial }{\partial t} \Phi_{x_i}(x, t) = DX( \Phi(x, t) )\Phi_{x_i}(x, t)  \qquad \Phi_{x_i}(x, 0) = 1.
\]
Together with the previous estimate this implies
\[
| \Phi_{x_i} (x,t)- 1 | \leq  C \delta \qquad \forall\,x\in M
\]
for every $i = 1, 2, \dots, n$ since $p>n$. We differentiate \eqref{flow} once more and use the previous estimates to obtain $\| \Phi_{x_ix_j} (\cdot,t) \|_{L^p(M)} \leq  C \delta$, which implies  \eqref{almost_id}.

It follows from \eqref{almost_id} that  $|\nu_M(x)- \nu_{M_t}(\Phi(x,t))| \leq C \delta$ for $x \in M$ and $t \in [0,1]$. Therefore \eqref{reg_Y} implies  
\begin{equation} \label{trivial_est}
|Y(x) - \nu_{M_t}(x)| \leq \frac{1}{8}\qquad \forall\,  x \in M_t.
\end{equation}
We recall that $T_F(\cdot)$  is constant along the trajectories, i.e., $T_F(\Phi(x,t))= T_F(x)$ for $x \in M$.  Hence, by \eqref{reg_Y}   we have  $|X(\Phi(x,t))| \leq C |T_F(x)| \leq C |\langle X(x), \nu_M(x) \rangle|$ for $x \in M$. Similarly, it follows from \eqref{trivial_est} that 
\begin{equation} \label{easy_est}
C^{-1}  |\langle X(x), \nu_{M}(x) \rangle | \leq |\langle X(\Phi(x,t)), \nu_{M_t}(\Phi(x,t)) \rangle |  \leq C  |\langle X(x), \nu_{M}(x) \rangle |  \qquad \forall\, x \in M.
\end{equation}
for every $t \in [0,1]$, when $\delta$ is sufficienlty small.

We proceed by  showing \eqref{symm_diff}.  We use the same calculation as in the proof of Theorem \ref{gruter} to estimate
\[
\begin{split}
\frac{d }{d t} |E_t \triangle E| \leq   \int_{M_t }| \langle X, \nu_{M_t} \rangle |\, d \Ha^{n-1} \leq C \int_{M}| \langle X, \nu_{M} \rangle |\, d \Ha^{n-1} 
\end{split}
\]
where the last inequality follows from \eqref{almost_id}  and \eqref{easy_est}. Hence
\[
|F \triangle E|= \int_{0}^1 \frac{d }{d t} |E_t \triangle E| \, dt \leq C  \int_{M}| \langle X, \nu_{M} \rangle |\, d \Ha^{n-1}. 
\]

The  estimate \eqref{nasty_estimate}   follows exactly as \cite[Lemma 7.1]{AFM}, but we give the proof for the convenience of the reader. We notice that  \eqref{easy_est} implies $|X_{\tau_t}| \leq |X| \leq C|\langle X, \nu_{M_t} \rangle |$ on $M_t$. In particular, $\| X_{\tau_t}\|_{L^2(M_t)} \leq C \| \langle X, \nu_{M_t} \rangle\|_{L^2(M_t)}$. Notice that we may write $X = \langle X, \bar{Y}\rangle Y$, where $\bar{Y} = \frac{Y}{|Y|}$. Estimates \eqref{bound_Y} and  \eqref{trivial_est} imply $|\bar{Y}- \nu_{M_t} |\leq \frac{1}{4}$ on $M_t$. We may estimate the tangential differential on $M_t$ by \eqref{bound_Y},  \eqref{trivial_est} and by the $C^2$-regularity of $Y$ and $\bar{Y}$ 
\[
\begin{split}
|D_{\tau_t} X_{\tau_t}| &= |D_{\tau_t} X - D_{\tau_t} (\langle X, \nu_{M_t}\rangle  \nu_{M_t})| = |D_{\tau_t} (\langle X, \bar{Y}\rangle Y) - D_{\tau_t} (\langle X, \nu_{M_t}\rangle  \nu_{M_t})| \\
&\leq |D_{\tau_t} ( \langle X, \nu_{M_t} \rangle (Y - \nu_{M_t} )| + |D_{\tau_t}  (\langle X, (\bar{Y} - \nu_{M_t})\rangle Y)|\\
&\leq \frac{1}{8} |D_{\tau_t} \langle X, \nu_{M_t} \rangle| +  \frac{1}{2}|D_{\tau_t}X| + C |X| (1+ |D_{\tau_t}\nu_{M_t}  |).
\end{split}
\]
Since $|D_{\tau_t}X| \leq |D_{\tau_t} X_{\tau_t}| + |D_{\tau_t} \langle X, \nu_{M_t} \rangle|+ |X||D_{\tau_t}\nu_{M_t}  |$ we obtain
\[
|D_{\tau_t} X_{\tau_t}| \leq C|D_{\tau_t} \langle X, \nu_{M_t} \rangle| + C  |\langle X, \nu_{M_t} \rangle|  (1 + |D_{\tau_t}\nu_{M_t}  |),
\]
where we have also used \eqref{easy_est}. We integrate this and use  H\"older's inequality to get
\[
\begin{split}
\|D_{\tau_t} X_{\tau_t}\|_{L^2(M_t)}^2 &\leq C\|\langle X, \nu_{M_t} \rangle\|_{H^1(M_t)}^2  + C \int_{M_t}  |\langle X, \nu_{M_t} \rangle|  |D_{\tau_t}\nu_{M_t}  | \, d \Ha^{n-1}\\
&\leq  C\|\langle X, \nu_{M_t} \rangle\|_{H^1(M_t)}^2  + C\|\langle X, \nu_{M_t} \rangle\|_{L^{\frac{2p}{p-2}}(M_t)}^2  \||D_{\tau_t}\nu_{M_t}  \|_{L^p(M_t)}^2\\
&\leq  C\|\langle X, \nu_{M_t} \rangle\|_{H^1(M_t)}^2, 
\end{split}
\]
where the last inequality follows from \eqref{almost_id} and from Sobolev inequality with $p>n$.
\end{proof}

In the following lemma we show, by continuity, that if $J$ has positive second variation at a critical point $E$ then the quadratic form \eqref{quadratic} remains positive in a $W^{2,p}$-neighborhood of $E$.  
\begin{lemma}
\label{continuity_2nd}
Let $E$ be as in Theorem \ref{mainthm} and $p> n$. There exists $\delta>0$ such that for any $C^2$-regular set $F$ with $\| E,F \|_{W^{2,p}} \leq \delta$ and every $\varphi \in H^1(M_F)$ with  $\int_{M_F} \varphi \, d \Ha^{n-1} = 0$ it holds
\[
\partial^2 J(F)[\varphi] \geq \frac{c_0}{2}\|\varphi \|_{H^1(M_{F})}^2.
\] 
Here $M_F = \overline{\partial F \cap \Omega}$,  $\partial^2 J(F)[\varphi] $ is defined in \eqref{quadratic} and $c_0$ is the constant from Lemma \ref{easy_lemma}. 
\end{lemma}

\begin{proof}
We argue by contradiction and assume that there are $F_k$ with $\| E,F_k \|_{W^{2,p}}=\eps_{k} \to 0$ and $\varphi_k \in  H^1(M_k)$ with $\int_{M_k} \varphi_k \, d \Ha^{n-1}= 0$, which by scaling we may assume to satisfy $\|\varphi_k \|_{ H^1(M_k)}= 1$, such that
\begin{equation} \label{too_small}
\partial^2 J(F_k)[\varphi_k] < \frac{c_0}{2}
\end{equation}
where $M_k = \overline{\partial F_k \cap \Omega}$. Let us recall the formula for $\partial^2 J(F_k)$ 
\begin{equation} \label{2nd_for_F}
\begin{split}
\partial^2 J(F_k)[\varphi_k] = & \int_{M_k} \left( |D_{\tau_k}\varphi_k |^2- | B_{M_k} |^2  \varphi_k^2 \right)\, d\Ha^{n-1}  -  \int_{M_k \cap \partial \Omega}  B_{\partial \Omega}(\nu_k,\nu_k) \varphi_k^2  \,d \Ha^{n-2}\\
&+ 8 \gamma \int_{M_k}  \int_{M_k}  G(x,y) \varphi_k(x)  \varphi_k(y) \, d\Ha^{n-1}(x)d\Ha^{n-1}(y)\\
&+  4 \gamma \int_{M_k}  \langle \nabla v_{F_k}, \nu_k \rangle \varphi_k^2 \, d\Ha^{n-1},
\end{split}
\end{equation}
where $\nu_k$ is the unit normal of $F_k$. We will show that there exists $\varphi \in  H^1(M)$ with $\|\varphi\|_{ H^1(M)}=1$ and  $\int_{M} \varphi \, d \Ha^{n-1} = 0$ such that, up to subsequences, 
\[
\lim_{k \to \infty} \partial^2 J(F_k)[\varphi_k] \geq \partial^2 J(E)[\varphi]. 
\] 
This together with \eqref{too_small} contradicts Lemma \ref{easy_lemma}. 

Since $\| E,F_k \|_{W^{2,p}} \to 0$, we can find a sequence of $C^2$-diffeomorphisms $\Psi_k: E \to F_k$ such that $\|\Psi_k- Id\|_{W^{2,p}} \to 0$. By compactness there exists  $\varphi \in  H^1(M)$ with $\|\varphi\|_{ H^1(M)}=1$ and  $\int_{\partial E} \varphi \, d \Ha^{n-1} = 0$ such that, up to a subsequence,  
\[
\varphi_k \circ \Psi_k \rightharpoonup \varphi \qquad \text{weakly in }\,  H^1(M),
\]
where $M = \overline{\partial E \cap \Omega}$. In particular, $\varphi_k \circ \Psi_k \to \varphi$ strongly in $L^2(M)$. We also conclude that $\nu_k \circ \Psi_k \to \nu$  uniformly on $M$ where $\nu$ is the unit normal of $E$. Therefore for the first term in \eqref{2nd_for_F}, by the weak lower semicontinuity, we obtain 
\[
\lim_{k \to \infty} \int_{M_k} |D_{\tau_k}\varphi_k |^2\, d\Ha^{n-1} \geq  \int_{M} |D_{\tau}\varphi |^2\, d\Ha^{n-1} .
\]
Next we observe that the following convergences holds:
\[
B_{M_k} \circ \Psi_k \to B_M \quad \text{in } \, L^p(M) \quad \text{and } \quad v_{F_k} \to v_E \quad \text{in } \, C^{1, \alpha}(\Omega).
\]
Indeed, the first one follows immediately from the $W^{2,p}$-convergence of $F_k$ and the second one follows from the uniform $C^{1,\alpha}$-regularity given by the equation \eqref{Laplace}. Therefore we obtain the convergences of the second and the last term in \eqref{2nd_for_F}
\[
\lim_{k\to \infty} \int_{M_k}| B_{M_k} |^2  \varphi_k^2 \, d\Ha^{n-1} =  \int_{M}| B_{M} |^2  \varphi^2 \, d\Ha^{n-1}
\]
and 
\[
\lim_{k\to \infty} \int_{M_k}  \langle \nabla v_{F_k}, \nu_k \rangle \varphi_k^2 \, d\Ha^{n-1} =  \int_{M}  \langle \nabla v_{E}, \nu \rangle \varphi^2 \, d\Ha^{n-1}.
\]

By compactness of the trace operator we have that  $\varphi_k  \circ \Psi_k  \to \varphi$ on $L^2(M \cap \partial \Omega)$. Hence we deduce the convergence of the boundary term
\[
\lim_{k\to \infty}\int_{M_k \cap \partial \Omega}  B_{\partial \Omega}(\nu_k,\nu_k) \varphi_k^2  \,d \Ha^{n-2} = \int_{M \cap \partial \Omega}  B_{\partial \Omega}(\nu,\nu) \varphi^2  \,d \Ha^{n-2}.
\]

To conclude the proof it remains to show that
\begin{equation*}
\lim_{k\to \infty} \int_{M_k}  \int_{M_k}  G(x,y) \varphi_k(x)  \varphi_k(y) \, d\Ha^{n-1}d\Ha^{n-1} =  \int_{M}  \int_{M}  G(x,y) \varphi(x)  \varphi(y) \, d\Ha^{n-1}d\Ha^{n-1},
\end{equation*}
which is equivalent to show that 
\[
\lim_{k \to \infty}\|Dw_k\|_{L^2(\Omega)} = \|Dw\|_{L^2(\Omega)} 
\]
where 
\[
- \Delta w_k =  \varphi_k \Ha^{n-1} \lfloor_{M_k}, \qquad - \Delta w_k =  \varphi \Ha^{n-1} \lfloor_{M}.
\]
This in turn follows from the strong convergence of $\varphi_k \Ha^{n-1} \lfloor_{M_k}$ to $\varphi \Ha^{n-1} \lfloor_{M}$ in $H^{-1}(\Omega)$. The argument is the same as in the proof of \cite[Theorem 3.9, Step 1]{AFM} and therefore we present here only the main stepes outlining the differences. Fix $f \in H^1(\Omega) \cap C^1(\Omega)$.  Denote by $J_M\Psi_k$ the Jacobian of $\Psi_k$ on $M$ and notice that $J_M\Psi_k \to 1$ uniformly on $M$. For simplicity we denote $\tilde{\varphi}_k= \varphi_k \circ \Psi_k $. By Lemma \ref{the_flow} we can construct a flow $\Phi_{F_k}(\cdot, t)$ such that $\Phi_{F_k}(x, 0)= x$  and  $\Phi_{F_k}(E, 1)= F_k$, and that the associated vector field satisfies $\|X_{k}\|_{W^{2,p}}\leq C \eps_k$. With a slight abuse of notation we denote $\Phi_{F_k}(\cdot, 1) = \Phi_k$.

As in \cite{AFM} we have 
\[
\begin{split}
\left\langle (\varphi_k  \Ha^{n-1} \lfloor_{M_k} -  \varphi \Ha^{n-1} \lfloor_{M}), f \right\rangle = & \int_{M_k} f \varphi_k \, d \Ha^{n-1} -  \int_{M} f \varphi \, d \Ha^{n-1}\\
\leq & \,C \,\|\tilde{\varphi}_k\|_{L^2(M)}\cdot \| f \circ \Phi_k - f\|_{L^2(M)} +\tilde{\eps}_k \|f\|_{H^1(\Omega)} \\
&+  C\|J_M\Phi_k - 1\|_{L^{\infty}(M)} 
\|f\|_{H^1(\Omega)},
\end{split}
\]
with $\tilde{\eps}_k \to 0$. Using \eqref{easy_est} and arguing  as  in the proof of Theorem \ref{gruter} we may estimate pointwise the n-dimensional Jacobian of $\Phi_{F_k}: M \times [0,1] \to \Omega$ from below by $c|X_{k}(\Phi_{F_k}(x,t))|$ with $c>0$.  Let us denote $M_0 = M \cap \{ X_{k} \neq 0\}$. Then \eqref{easy_est} implies that the  map $\Phi_{F_k}: M_0 \times [0,1] \to \Omega$ is  diffemorphism onto its image and that $|X_{k}(\Phi_{F_k}(x,t))| = 0$ for $x \in M \setminus M_0$ and $t \in [0,1]$. Therefore  we have
\[
\begin{split}
\| f \circ \Phi_k - f\|_{L^2(M)} ^2 &= \int_{M} |f (\Phi_k(x)) - f(x)|^2 \, d \Ha^{n-1} \\
&=  \int_{M} \left| \int_0^1 \frac{d}{dt}f (\Phi_{F_k}(x,t)) \,dt \right|^2 d \Ha^{n-1}\\
&\leq  \int_{M}  \int_0^1 |\nabla f (\Phi_{F_k}(x,t))|^2 |X_{k}(\Phi_{F_k}(x,t))|^2 \,dt\, d \Ha^{n-1}\\
&\leq  \|X_{k}\|_{L^{\infty}} \int_{M_0}  \int_0^1 |\nabla f (\Phi_{F_k}(x,t))|^2 |X_{k}(\Phi_{F_k}(x,t))| \,dt \,d \Ha^{n-1}\\
&\leq C  \|X_{k}\|_{L^{\infty}} \int_{\Omega}  |\nabla f (y)|^2  \,dy.
\end{split}
\]
Hence we have 
\[
\|\varphi_k  \Ha^{n-1} \lfloor_{M_k} -  \varphi \Ha^{n-1} \lfloor_{M}\|_{H^{-1}(\Omega)} \to 0, 
\]
and the proof is completed.
\end{proof}

We are now ready to prove Proposition  \ref{smooth}.
\begin{proof}[\textbf{Proof of the Proposition \ref{smooth}}]
Without loss of generality we may assume that $F$ is $C^2$-regular  and $M_F= \overline{\partial F \cap \Omega}$ is a $C^2$-manifold with boundary that meets  $\partial \Omega$ orthogonally. 

Let  $X$ be the vector field and $\Phi$ be the associated flow given by Lemma \ref{the_flow}.  For every $t \in [0,1]$ we set $E_t = \Phi(E, t)$, $M_t = \overline{\partial E_t \cap \Omega}$ and denote by $\nu_t$ the unit normal to $M_t$. Since $E$ is a critical set we have
\[
\frac{d}{dt} J(E_t) \big|_{t=0} = 0.
\] 
Moreover, Proposition \ref{2ndVariation} and Lemma \ref{easy_lemma} yield in particular, for $\varphi =  \langle X, \nu \rangle$ that 
\[
\frac{d^2}{dt^2} J(E_t) \big|_{t=0} = \partial^2 J(E)[\varphi] \geq c_0 \|\varphi\|_{H^{1}(M)}^2
\]  
for some constant $c_0>0$. The idea of the proof is to show, by a continuity argument, that if $\delta>0$ is chosen small enough, we have
\begin{equation}
 \label{key_estim}
\frac{d^2}{dt^2} J(E_t) \geq  \frac{c_0}{4}\|\varphi_t\|_{H^{1}(M_t)}^2,
\end{equation}
for every $t \in [0,1]$, where $\varphi_t =  \langle X, \nu_t \rangle$. Once we have \eqref{key_estim} the result then follows since we may infer, by  \eqref{almost_id} and \eqref{easy_est}, that 
\[
\int_{M_t} |\varphi_t|^2 \, d \Ha^{n-1}  \geq c \int_{M}  |\langle X, \nu \rangle |^2 \, d \Ha^{n-1}.
\]
and therefore the estimates \eqref{symm_diff}  and \eqref{key_estim} yield
\[
J(F)- J(E) = \int_0^1 (1-t )\frac{d^2}{dt^2}J(E_t)\, dt \geq c\, \|\varphi\|_{L^2(M)}^2 \geq  c \, | F \triangle E|^2,
\]
proving the claim. 

Hence, it remains to prove \eqref{key_estim}. By proposition \ref{2ndVariation} we can write the second second derivative of $J(E_{t})$ at $t$ as 
\[
\begin{split}
\frac{d^2 J(E_t)}{dt^2}  = & \; \partial^2 J(E_t)[\varphi_t] -\int_{M_t} (H_{M_t} + 4 \gamma v_{E_t})\diver_{\tau_t} \left( X_{\tau_t} \langle X, \nu_{t} \rangle \right) \, d \Ha^{n-1} \\
&+ \int_{M_t} (H_{M_t} + 4 \gamma v_{E_t})\diver(X)  \langle X, \nu_t \rangle  \, d \Ha^{n-1} 
\end{split}
\]
for $\varphi_t =  \langle X, \nu_t \rangle$ as before.  Notice that since the flow preserves the volume we have
\[
 \int_{M_t} \varphi_t \, d \Ha^{n-1}  =  \int_{M_t} \langle X, \nu_t \rangle \, d \Ha^{n-1} =0.
\]
Therefore when $\delta>0$ is small enough Lemma \ref{continuity_2nd} yields
\begin{equation} 
\label{1_terms_small}
\partial^2 J(E_t)[\varphi_t] \geq  \frac{c_0}{2}\|\varphi_t \|_{H^1(M_t)}^2.
\end{equation}
Moreover, by \eqref{def_Phi_X} $\diver X= 0$ and therefore to conclude the proof it is enough to show that, when $\delta>0$ is choosen small enough, we have
\begin{equation} \label{1_extraterm}
S_t:= \left| \int_{M_t} (H_{M_t} + 4 \gamma v_{E_t})\diver_{\tau_t} \left( X_{\tau_t} ( \langle X, \nu_{t} \rangle) \right) \, d \Ha^{n-1} \right| \leq  \frac{c_0}{4}\|\varphi_t \|_{H^1(M_t)}^2.
\end{equation}
Since $E$ solves \eqref{euler_lag} we deduce that $(H_{M_t} + 4 \gamma v_{E_t})\circ \Phi(\cdot, t) \to \lambda$ in $L^p(M)$. Therefore given $\eps>0$ for sufficiently small $\delta$we have 
\[
\begin{split}
S_t &=  \big| \int_{M_t} (H_{M_t} + 4 \gamma v_{E_t}- \lambda)\diver_{\tau_t} \left( X_{\tau_t}  \langle X, \nu_{t} \rangle \right) \, d \Ha^{n-1} \big| \\
&\leq  \left( \int_{M_t} (H_{\partial F_t} + 4 \gamma v_{E_t}- \lambda)^p\, d \Ha^{n-1} \right)^{\frac{1}{p}} \left( \int_{M_t} \left( \diver_{\tau_t} ( X_{\tau_t} \langle X, \nu_{t} \rangle) \right)^{\frac{p}{p-1}}\, d \Ha^{n-1} \right)^{\frac{p-1}{p}} \\
&\leq \eps \left(  \|D_{\tau_t} X_{\tau_t}\|_{L^2(M_t)} \| \langle X, \nu_{t} \rangle\|_{L^{\frac{2p}{p-2}}(M_t)} +    \|D_{\tau_t} \langle X, \nu_{t} \rangle\|_{L^2(M_t)} \| X_{\tau_t}\|_{L^{\frac{2p}{p-2}}(M_t)} \right) \\
&\leq C \eps\,  \|X_{\tau_t}\|_{H^1(M_t)} \| \langle X, \nu_{t} \rangle\|_{H^1(M_t)}
\end{split}
\]
where the last inequality follows from Sobolev inequality and from $p>n$. The estimate \eqref{1_extraterm} then follows from \eqref{nasty_estimate} of Lemma \ref{check_reg}.
\end{proof}

\section{Proof of the main theorem}

In this section we prove Theorem  \ref{mainthm}. As it was  mentioned in the introduction we will use the regularity of $\Lambda$-minimizers  to rule out those competing sets which are not regular. This idea goes back to the work by White \cite{Wh} and more recently it has been used by Cicalese and Leonardi \cite{CL} and by  Fusco and Morini \cite{FM}. We begin by proving a simple lemma.
\begin{lemma} \label{isop.lemma}
Suppose that $E \subset \Omega$ is $C^2$-regular and satisfies the orthogonality condition \eqref{orthogonality}. There is a constant $C$ depending only on $E$ such that  for every $ F \in  BV(\Omega)$ we have
\[
J(E) \leq J(F) +  C | F \triangle E|\,.
\]
\end{lemma}

\begin{proof}
By the Lipschitz continuity of the non-local part \eqref{lipschitz} we have
\[
J(E) \leq J(F ) +  P(E, \Omega) - P(F, \Omega) + C\, | F \triangle E|\, .
\]
On the other hand by the assumptions on $E$ we may construct a $C^1$ vector field $X $ on $\Omega$ such that  $X=\nu_E$ on $M = \overline{\partial E \cap \Omega}$, $ \langle X, \nu_{\Omega} \rangle = 0 $ on $\partial \Omega$  and $\|X\|_{L^{\infty}} \leq 1$.  Therefore
\[
\begin{split}
 P(E, \Omega) - P(F, \Omega) &\leq \int_{\partial (E \cap \Omega)} \langle X, \nu \rangle\, d\Ha^{n-1}(x) - \int_{\partial (F \cap \Omega)} \langle X, \nu \rangle\,d \Ha^{n-1}(x) \\
&= \int_{E \cap \Omega} \diver  X\, dx - \int_{F \cap \Omega} \diver X\, dx \\
&\leq  \|\diver  X\|_{L^{\infty}} | F \triangle E|\, .
\end{split}
\]
\end{proof}
We remark that the above proof also yields that  $E$ is a $\Lambda$-minimizer in the sense of Definition \ref{almost_area}.

The next lemma that relates convergence of mean curvatures with $W^{2,p}$-convergence, is very similar to  \cite[Lemma 7.2]{AFM} therefore we only give a sketch of the proof.

\begin{lemma} 
\label{elliptic}
Let $p >n$. Suppose that $E$ is as in Theorem \ref{mainthm}, and let $F_{k}$ be a sequence of sets such that there exist diffeomorphisms $\Psi_k: E \to F_k$ such that  $\|\Psi_k - Id\|_{C^{1,\alpha}} \to 0$ and
\[
H_{\partial F_k}  (\Psi_k (\cdot)) \to H_{\partial E} (\cdot)  \quad \text{in }\, L^p(M)\,.
\]
Then we can find a family of diffeomorphisms $\tilde{\Phi}_k: E \to F_k$ such that  $\|\tilde{\Phi}_k - Id\|_{W^{2,p}} \to 0$. 
\end{lemma} 

\begin{proof}[Sketch of the proof]
Since $M_k$ are $C^{1,\alpha}$ manifolds with boundary they can be represented locally as a graph of $C^{1,\alpha}$ functions. Far from the relative boundary $\overline{\partial F_k \cap \partial\Omega}$ the convergence follows as in  \cite[Lemma 7.2]{AFM}. On the boundary we can flatten the boundary $\partial \Omega$ locally to a half space and then use  similar elliptic regularity estimate as  in the interior case. 
\end{proof}

We are now ready to prove the main result of the paper. 

\begin{proof}[Proof of Theorem \ref{mainthm}] As in \cite{AFM,CJP} the proof is divided into two steps: we first prove that $E$ is a local minimizer with respect to Hausdorff distance and then sharpen the result to obtain the local minimality with respect to $L^1$ distance  with the quantitative estimate.

\vspace{0.3cm}
\noindent\textsc{Step 1:}  First we prove that $E$ is a  local minimizer in $L^{\infty}$-topology, i.e., there is $\delta>0$ such that for every $F \subset \mathcal{I}_{\delta}(E):=  \{ x \in \Omega \mid \dist(x,  E) < \delta \}  $ of finite perimeter in $\Omega$ with $|F|= |E|$ it holds
\[
J(F)\geq J(E).
\] 

We argue by contradiction and assume that there exist $E_k  \subset \Omega$ such that $|E_k|= |E|$ , $E_k \subset \mathcal{I}_{1/k}(E )$ verifying
\[
J(E_k) < J(E).
\]
Since $M=\overline{\partial E \cap \Omega}$ is  $C^{4, \alpha}$-regular and satisfies the orthogonality condition \eqref{orthogonality}, there exist a neighborhood of $M$,  $U$, and a function  $g_M \in C^{2}(U)$ such that 
\begin{enumerate}
\item[(i)] $g_{M} = 0$ on $M$ and $g_{M} \leq  0$ in $E$,
\item[(ii)] $|\nabla g_{M}| \geq c >0$ in  $U$,  
\item[(iii)] $\langle \nabla g_{M},  \nu_{\Omega} \rangle    = 0 $ on $ \partial \Omega \cap U$,
\end{enumerate}
We denote the sublevel set  of $g_{M} $ by
\[
U_{\eps} := \{ x \in \Omega \mid g_{M}(x) < \eps\} .
\]
These sets are  $C^2$ regular and, by the previous condition (iii), they satisfy the orthogonality condition \eqref{orthogonality}. Moreover,  there is a sequence $\eps_k \to 0$ such that 
\begin{itemize}
\item[(i)] $ \mathcal{I}_{1/k}(E ) \subset   U_{\eps_k}$,
\item[(ii)]  $\|U_{\eps_k} , E\|_{C^{2}} \to 0$. 
\end{itemize}
We remark that we may, for instance, define the sets $U_{\eps}$ by using the flow $\Phi_Y$ constructed in the Step 2 of Lemma \ref{the_flow} as
\[
U_{\eps}:= \Phi_Y(E, \eps). 
\]
Using the termonology of Lemma \ref{the_flow} we choose $T_F(x) =1$  for every $x$. The orthogonality of $U_{\eps}$, then follows exactly as in Step 3 of Lemma \ref{the_flow}. 

Using an argument similar to \cite[proof of Theorem 4.3, Step 1]{AFM} we show that the contradicting sequence $\{E_k\}_{k}$ can be replaced by $\{F_k\}_{k}$ defined by the minimizers of the problems
\begin{equation} \label{replace}
J(F) + \Lambda_1 \big| |F|- |E| \big|,  \qquad  F \subset U_{\eps_k},
\end{equation}
when $\Lambda_1$ is sufficiently large. This will follows once we prove that we can choose $\Lambda_{1}$ in such a way that $F_k$ satisfies the volume constraint
\begin{equation}
\label{volume_con2}
|F_k| = |E|.
\end{equation}
Indeed, \eqref{volume_con2} together with the minimality of $F_{k}$ for \eqref{replace} give $J(F_k) < J(E)$ for every $k$ and obviously $F_k \to E$ in Hausdorff topology.

We then argue by contradiction and assume that $|F_k| < |E|$, the case $|F_k| > |E|$ can be handle similarly. We may choose  $\eps \leq \eps_k$ such that the set 
\[
\tilde{F_k} = F_k \cup U_{\eps}
\]
satisfies $|\tilde{F_k}| = |E|$. The reduced boundary $\partial^* \tilde{F_k}$ can be decomposed in three disjoint parts, one contained in $\partial^* \tilde{F_k} \setminus \partial U_{\eps}$, another contained in $\partial U_{\eps} \setminus \partial^* \tilde{F_k}$, and the third one $\{ x \in \partial^* \tilde{F_k} \cap  \partial U_{\eps} \mid \nu_{\tilde{F_k}}(x)= \nu_{U_{\eps}}(x) \}$. Notice that  $\nu_{U_{\eps}} = \frac{\nabla g_{M}}{ | \nabla g_{M}|}$ on $\partial U_{\eps}$. Therefore by choosing  $X :=  \frac{\nabla g_{M}}{ | \nabla g_{M}|}$ we get
\[
P(\tilde{F_k} , \Omega) - P(F_k , \Omega) \leq  \int_{\partial^* \tilde{F_k} \cap \Omega} \langle X, \nu_{\tilde{F_k}} \rangle\, d \Ha^{n-1} - \int_{\partial^* F_k \cap \Omega} \langle X, \nu_{\F_k} \rangle\, d \Ha^{n-1}. 
\]
Notice that it holds $F_k \subset \tilde{F_k}$. Hence, by the Lipschitz property \eqref{lipschitz} we have
\[
\begin{split}
J(\tilde{F_k}) - J(F_k) - \Lambda_1 \big| |F_k|- |E| \big|  \leq & \,  P(\tilde{F_k} , \Omega) - P(F_k , \Omega) + (C\gamma -\Lambda_1)|\tilde{F_k} \setminus F_k|\\
 \leq  &  \int_{\partial^* \tilde{F_k} \cap \Omega} \langle X, \nu_{\tilde{F_k}} \rangle\, d \Ha^{n-1} - \int_{\partial^* F_k \cap \Omega} \langle X, \nu_{F_k} \rangle\, d \Ha^{n-1} \\ 
 & + (C\gamma -\Lambda_1)|\tilde{F_k} \setminus F_k|\\
\leq & \int_{\tilde{F_k} \setminus F_k} |\diver X|\, dx  + (C\gamma -\Lambda_1)|\tilde{F_k} \setminus F_k|.
\end{split}
\]
Therefore \eqref{volume_con2} follows if 
\[
\Lambda_1 > C\gamma + \|\diver X\|_{L^{\infty}}.
\]

Next we show that $F_k$ are $\Lambda$-minimizers with a constant $\Lambda$, independent of $k$. To that aim let $G \subset \Omega $ be a  set of finite perimeter. We divide $G$ into two parts
\[
G \cap U_{\eps_k}\quad \text{and} \quad  G \setminus U_{\eps_k}.
\] 
By the minimality of $F_k$ and by \eqref{lipschitz} we obtain
\begin{equation}
\label{inside}
P(F_k, \Omega) \leq P(G \cap U_{\eps_k},  \Omega) + C\gamma|(G \triangle F_k) \cap U_{\eps_k}|.
\end{equation}
Moreover, since $U_{\eps_k}$ are uniformly $C^2$ regular and satisfy the orthogonality condition, they are $\Lambda$-minimizers, as we remarked after the proof of Lemma \ref{isop.lemma}. Hence, 
\begin{equation}
\label{outside}
P(U_{\eps_k},  \Omega) \leq P(G \cup U_{\eps_k},  \Omega) +C|G \setminus U_{\eps_k}|\,,
\end{equation}
for some $C>0$. Since 
\[
P(G \cup U_{\eps_k},  \Omega) + P(G \cap U_{\eps_k},  \Omega) \leq P(G,  \Omega) + P(U_{\eps_k},  \Omega)
\] 
the estimates \eqref{inside} and \eqref{outside} yield
\[
P(F_k,  \Omega) \leq P(G,  \Omega) +  \Lambda |F_k \triangle G|
\] 
for some large  $\Lambda$. Thus we have the  $\Lambda$-minimality. By the $\Lambda$-minimizing property of $F_k$ and by Theorem \ref{convergence} we conclude that $F_k \to E$ in $C^{1,\alpha}$ and that $F_k$ satisfy the orthogonality condition \eqref{orthogonality}.

Finally we will use the Euler-Lagrange equation for $F_k$ to conclude that $\|F_k, E\|_{W^{2,p}} \to 0$. The Euler-Lagrange equation for $F_k$ reads as 
\begin{equation}
\label{euler_pen1}
 \begin{cases}
 H_{M_k}  + 4 \gamma v_{F_k} = \lambda_k &\text{ on } M_k \cap  U_k, \\
 H_{M_k}  + 4 \gamma v_{E} + \gamma_k = \lambda  &\text{ on } M_k \setminus  U_k,
\end{cases}
\end{equation}
where $\gamma_k$ is some remainder term which converges uniformly to zero, and $\lambda_k$ and $\lambda$ are the Lagrange multipliers associated to the volume constraint. We remark also that the proof of the Theorem \ref{gruter} implies that the mean curvatures of the relative boundaries of the sets $F_k$ are uniformly bounded, i.e.  $\|H_{M_k}\|_{L^{\infty}} \leq \Lambda$.  Moreover the equation \eqref{Laplace}  implies  
\begin{equation}
\label{non-local_conv}
v_{F_k} \to v_E \qquad \text{ in } \,  C^{1}(\Omega).
\end{equation}

We can then show that 
\begin{equation}
\label{curv_conv}
H_{M_k}  (\Psi_k (\cdot)) \to H_{M} (\cdot)  \quad \text{ in }\, L^p(M)\,,\;\forall\;p>n.
\end{equation}
Indeed, consider the vector field $X \in C^{1}(\Omega, \R^n)$ as in the proof of Lemma \ref{isop.lemma}, i.e., $X = \nu_M$ on $M$ and $ \langle X, \nu_{\Omega} \rangle = 0 $ on $\partial \Omega$. We multiply the equation \eqref{euler_pen1} by $\langle X, \nu_{M_k} \rangle$, integrate over $M_k$ and use integration by parts (recall that $M_k$ satisfies the orthogonality condition \eqref{orthogonality}) to deduce 
\[
\begin{split}
\int_{M_k \cap  U_k}( \lambda_k - 4 \gamma v_{F_k}  ) &\langle X, \nu_{M_k} \rangle\, d \Ha^{n-1}  + \int_{M_k \setminus  U_k} ( \lambda -  4 \gamma v_{E} - \gamma_k  )\langle X, \nu_M \rangle\, d \Ha^{n-1}\\
&=  \int_{M_k \cap \Omega} H_{M_k}  \langle X, \nu_{F_k} \rangle  \, d \Ha^{n-1} \\
&= \int_{M_k \cap \Omega} \diver_{\tau_k} X \, d \Ha^{n-1},
\end{split}
\]
Moreover using the  $C^{1,\alpha}$-convergence we get
\begin{equation*}
\lim_{k\to \infty}\int_{M_k \cap \Omega} \diver_{\tau_k} X \, d \Ha^{n-1} = \int_{M  \cap \Omega} \diver_{\tau} X\, d \Ha^{n-1}=  \int_{M\cap \Omega} ( \lambda -  4 \gamma v_{E} )\, d \Ha^{n-1},
\end{equation*}
where the last equality follows from the Euler-Lagrange equation \eqref{euler_lag} and from the fact that $X = \nu_E$ on $\partial E \cap \Omega$. Therefore the $C^{1,\alpha}$-convergence of $F_k$ and \eqref{non-local_conv}   imply that either  $\lambda_k \to \lambda$, or $\Ha^{n-1}(M_k \cap  U_k) \to 0$. In either case we obtain \eqref{curv_conv} due to the fact that $H_{M_k}$ are uniformly bounded in $L^{\infty}$.

From \eqref{curv_conv} and Lemma \ref{elliptic} we deduce $\|F_k, E\|_{W^{2,p}} \to 0$. Since $F_k$ satisfy the orthogonality condition  we may use Proposition \ref{smooth} to conclude 
\[
J(F_k) \geq J(E)
\]
when $k$ is large. This contradicts the minimality  of $F_k$ since, as we already observed, we have $J(F_k) < J(E)$.

\vspace{0.3cm}
\noindent\textsc{Step 2:} As in the previous Step we argue by contradiction and assume that there exist $E_k  \subset \Omega$ such that $|E_k| = |E|$, $|E_k \triangle E| \to 0$ and
\[
J(E_k) < J(E) + \frac{c_1}{4}\, |E_k \triangle E|^2 \, ,
\]
where the constant  $c_1$ is from  Proposition \ref{smooth}. Denote $\eps_k := |E_k \triangle E|$. We will replace the contradicting sequence $\{E_k\}_{k}$ by $\{F_k\}_{k}$, were each $F_k$ solves the minimization problem 
\begin{equation} \label{replace0}
\min \left\{ J(F)   + \Lambda_1 \sqrt{  (|F \triangle E| -\eps_k )^2 + \eps_k} \, : \,  F \subset \Omega\,\, \text{with }\, |F| = |E| \right\}, 
\end{equation}
for some constant $\Lambda_1$ which will be chosen later.

We may use the same argument  as in the proof of  Proposition \ref{almost_area_2} to deduce that $F_k$ minimizes the penalized problem
\begin{equation} \label{replace2}
J(F)   +  \Lambda_2 \big| |F| - |E|\big| +  \Lambda_1 \sqrt{  (|F \triangle E| -\eps_k )^2 + \eps_k},  \qquad F \subset \Omega, 
\end{equation}
for large enough $\Lambda_2$, which is   independent of  $k$.

By compactness we may assume that, up to a subsequence, $F_k \to F_0$ in $L^1$, and that $F_0$ minimizes
\[
J(F)   +    \Lambda_2  \big| |F| - |E|\big| +  \Lambda_1  |F \triangle E|,  \qquad F \subset \Omega.
\] 
By choosing $\Lambda_1$ large, but independent of $\Lambda_2$, it follows from Lemma \ref{isop.lemma} that $F_0= E$. In particular $F_k \to E$ in $L^1$. 

As in Step 1 we observe that every $F_k$ is a $\Lambda$-minimizer  with  $\Lambda$  independent of $k$. In fact, since there are no obstacle in \eqref{replace2}, this observation follows exactly as in the proof of  Proposition \ref{isop.lemma}. Therefore Theorem \ref{convergence} implies that $F_k \to E$ in $C^{1,\alpha}$ and that $F_k$ are $C^{1,\alpha}$-manifolds with boundary for sufficiently large $k$  and  satisfies the orthogonality condition \eqref{orthogonality}.  Moreover, since the mean curvature $H_{M_k}$  is bounded by $\Lambda$ we conclude that $M_k$ is $W^{2,p}$-regular for every $p>n$. 

We use  the minimality of $F_k$, the contradiction assumption, and Step 1 to obtain
\[
\begin{split}
J(F_k) + \Lambda_2 \sqrt{  (|F_k \triangle E| -\eps_k )^2 + \eps_k} &\leq J(E_k) + \Lambda_2 \sqrt{\eps_k} \\
&\leq  J(E) + \frac{c_0}{8}\eps_k^2  + \Lambda_2 \sqrt{\eps_k} \\
&\leq J(F_k) + \frac{c_0}{8}\eps_k^2  + \Lambda_2 \sqrt{\eps_k}.
\end{split}
\]
The previous inequality yields
\begin{equation}
\label{quotent_est}
\lim_{k \to \infty}  \frac{(|F_k \triangle E|- \eps_k) }{\eps_k} =0. 
\end{equation}
In particular
\begin{equation}
\label{quotent_est2}
  \frac{(|F_k \triangle E|- \eps_k) }{\sqrt{  (|F_k \triangle E| -\eps_k )^2 + \eps_k}  }  \leq \sqrt{\eps_k},
\end{equation}
for large $k$.

Arguing similarly to the proof of \cite[Theorem 1.1]{AFM} we now show that $M_k$ solves the Euler-Lagrange equation 
\begin{equation}
\label{euler_pen2}
 H_{M_k}  + 4 \gamma v_{F_k} +  \Lambda_1 f_k= \lambda_k  \quad \text{ on } M_k, 
\end{equation}
with $\|f_k\|_{L^{\infty}} \leq C \sqrt{\eps_k}$. Fix a vector field $X \in C^{\infty}_0(\Omega)$  such that $\diver X = 0$, and let $\Phi$ be the assciated flow. By \eqref{1st_vol} we have $|\Phi(F_k, t)| = |F_k|$. As in the proof of Theorem \ref{gruter} we conclude that 
\[
|\Phi(F_k, t) \triangle F_k| \leq |t| \int_{M_k} |\langle X, \nu_{M_k} \rangle|\, d \Ha^{n-1} + o(t).
\]
We then  use \eqref{quotent_est2} to estimate 
\begin{equation}
\label{quotent_est3}
\lim_{t \to 0} \frac{1}{t} \left| \sqrt{  \left(|\Phi(F_k, t) \triangle E| -\eps_k \right)^2 + \eps_k} - \sqrt{  \left(|F_k \triangle E| -\eps_k \right)^2 + \eps_k} \right| \leq C\sqrt{\eps_k}\int_{M_k} \left|\langle X, \nu_{M_k} \rangle\right|\, d \Ha^{n-1} .
\end{equation}
The minimality of $F_k$ yields
\[
J\left(\Phi(F_k, t)\right) + \Lambda_1 \sqrt{ \left(|\Phi(F_k, t) \triangle E| -\eps_k \right)^2 + \eps_k} \geq J(F_k) + \Lambda_1 \sqrt{  \left(|F_k \triangle E| -\eps_k \right)^2 + \eps_k}.
\]
Hence, we have by the first variation formula of $J(F_k)$ and by \eqref{quotent_est3} that 
\[
\left| \int_{M_k} (H_{M_k}  + 4 \gamma v_{F_k}) \langle X, \nu_{M_k} \rangle  \, d \Ha^{n-1} \right| \leq C \Lambda_1\sqrt{\eps_k}  \int_{M_k} |\langle X, \nu_{M_k} \rangle|\, d \Ha^{n-1}.
\]
By a density argument (cf. \cite[Corollary 3.4]{AFM}), the previous estimate implies
\[
\left| \int_{M_k} (H_{M_k}  + 4 \gamma v_{F_k}) \varphi  \, d \Ha^{n-1} \right| \leq C \Lambda_1\sqrt{\eps_k}  \int_{M_k} |\varphi|\, d \Ha^{n-1}
\]
for all $\varphi \in C_0^{\infty}(M_k)$ with $\int_{M_k} \varphi \, d \Ha^{n-1} = 0$.  By Riesz representation formula we then  obtain  \eqref{euler_pen2}.

 We use the equation \eqref{euler_pen2} and argue exactly as in Step 1 to conclude that  $\|F_k, E\|_{W^{2,p}} \to 0$  for any $p>n$. Hence we may use Proposition \ref{smooth} to conclude
\[
J(F_k) \geq J(E) + c_1|F_k \triangle E|^2
\] 
for large $k$. However, the minimality of $F_k$, the contradiction assumption, and \eqref{quotent_est}  yield
\[
J(F_k) \leq J(E_k) \leq  J(E) + \frac{c_1}{4}\eps_k^2 \leq J(E) + \frac{c_1}{2} |F_k \triangle E|^2
\]
for large $k$, which is a contradiction.
\end{proof}

\begin{proof}[Proof of Corollary \ref{corollary1}]
We obtain Corollary \ref{corollary1}  immediately from Theorem \ref{mainthm} arguing as in   \cite[Theorem 6.3]{AFM}.
\end{proof}

\section*{Acknowledgement}
The first author was partially funded by the 2008 ERC Grant no.226234 ''Analytic Techniques for Geometric and Functional Inequalities'' and the second author was partially funded by the Marie Curie project IRSES-2009-247486 of the Seventh Framework
Programme.


\begin{thebibliography}{100}

\bibitem{AFM}\textsc{E. Acerbi, N. Fusco \& M. Morini,} \emph{ Minimality via second variation for a non-local isoperimetric problem.} Preprint, 2011. 


\bibitem{AFP}\textsc{L. Ambrosio, N. Fusco \&  D. Pallara,} \emph{Functions of bounded variation and free discontinuity problems.} Oxford Press, 2000. 

\bibitem{CMM}
\textsc{F. Cangetti, M.G. Mora \& M. Morini,} 
\emph{A second order minimality condition for the Munford-Shah functional.} Calc. Var. Partial Differential Equations \textbf{33} (2008), 37-74. 

\bibitem{CJP}\textsc{G. M. Capriani, V. Julin and G. Pisante}, \emph {A quantitative second order minimality criterion for cavities in elastic bodies
}, Preprint. 

\bibitem{CR}\textsc{R. Choksi \& X. Ren}, \emph{On a Derivation of a Density Functional Theory for Microphase Separation of Di- block Copolymers}, J. Statist. Phys. \textbf{113} (2003), 151--176.

\bibitem{ChSt2}\textsc{R. Choksi \& P. Sternberg,} \emph{On the first and second variations of a non-local isoperimetric problem.} J. reine angew. Math. \textbf{611} (2007), 75-108. 

\bibitem{CL}\textsc{M. Cicalese \& G. Leonardi,} \emph{A selection principle for the sharp quantitative isoperimetric inequality.} Arch. Rat. Mech. Anal., \textbf{206} (2012), no.2, 617--643.

\bibitem{CS}\textsc{M. Cicalese \& E. Spadaro,}
\emph{Droplet Minimizers of an Isoperimetric Problem with long-range interactions.} arXiv, (2011).


\bibitem{FM}\textsc{N. Fusco \& M. Morini,} \emph{Equilibrium configurations of epitaxially strained elastic films: second order minimality conditions and qualitative properties of solutions.} Arch. Rational Mech. Anal., to appear.

\bibitem{GT}\textsc{D. Gilbarg \& N.S. Trudinger,} \emph{Elliptic partial differential equations of second order.} second edition, Springer, 1983. 

\bibitem{Giusti}
\textsc{E. Giusti},
\emph{Minimal Surfaces and Functions of Bounded Variations.} Birkh\"auser, 1994.



\bibitem{Gr}
\textsc{M. Gr\"uter}, 
\emph{Boundary regularity for solutions of a partitioning problem.}, Arch. Rational Mech. Anal. \textbf{97} (1987), no. 3, 261--270.

\bibitem{GJ}
\textsc{M. Gr\"uter \& J. Jost},
\emph{Allard type regularity results for varifolds with free boundaries.}, Ann. Scuola Norm. Sup. Pisa Cl. Sci. (4) \textbf{13} (1986), no. 1, 129--169.


\bibitem{NO}\textsc{Y. Nishiura \& I. Ohnishi}, 
\emph{Some mathematical aspects of the micro-phase separation in diblock copolymers}, Physica D 84:31--39 (1995).

\bibitem{OK}\textsc{T. Ohta \& K. Kawasaki},
\emph{Equilibrium morphology of block copolymer melts.}, Macromolecules 19 (1986), 2621--
2632.

\bibitem{RW1} \textsc{Ren X.; Wei J.}, \emph{ Concentrically layered energy equilibria of the di-block copolymer problem}. European J. Appl. Math. 13 (2002), 479--496.
\bibitem{RW2} \textsc{Ren X.; Wei J.}, \emph{  Stability of spot and ring solutions of the diblock copolymer equation}. J. Math. Phys. \textbf{45} (2004), 4106--4133.
\bibitem{RW3} \textsc{Ren X.; Wei J.}, \emph{  Wriggled lamellar solutions and their stability in the diblock copolymer problem}. SIAM J. Math. Anal. \textbf{37} (2005), 455--489.
\bibitem{RW4} \textsc{Ren X.; Wei J.}, \emph{  Many droplet pattern in the cylindrical phase of diblock copolymer morphology}. Rev. Math. Phys. \textbf{19} (2007), 879--921.
\bibitem{RW5} \textsc{Ren X.; Wei J.}, \emph{  Spherical solutions to a non-local free boundary problem from diblock copolymer morphology}. SIAM J. Math. Anal. \textbf{39} (2008), 1497--1535.




\bibitem{StZ}
\textsc{P. Sternberg \& K.  Zumbrun} \emph{ A Poincare inequality with applications to volume-constrained area-minimizing surfaces.} J. reine angew. Math. \textbf{503} (1998), 63--85.  

\bibitem{Tam}
\textsc{I. Tamanini}, 
\emph{Boundaries of Caccioppoli sets with H\"oder-continuous normal vector.}, J. Reine Angew. Math. \textbf{334} (1982), 27--39.

\bibitem{Wh} 
\textsc{White B.}, 
\emph{A strong minimax property of nondegenerate minimal submanifolds.}, J. Reine Angew. Math. \textbf{457}
(1994), 203--218.


\end{thebibliography}
\end{document}